\documentclass{amsart}

\usepackage{amsmath,amssymb,amsfonts,amsthm}
\usepackage{mathrsfs}
\usepackage[all]{xy}
\usepackage{url}
\usepackage{comment}

\newcommand{\R}{\mathbb{R}}

\newcommand{\N}{\mathbb{N}}

\newcommand{\X}{\mathcal{X}}
\newcommand{\Y}{\mathcal{Y}}
\newcommand{\U}{\mathcal{U}}
\newcommand{\st}{ \ \big| \ }

\newcommand{\A}{\mathcal{A}}

\newcommand{\Spec}[1]{\ensuremath {  \text{Spec} ( {#1}  )}}

\renewcommand{\and}{\text{and}}

\newcommand{\V}{\mathcal{V}}
\newcommand{\affin}[1]{ \ensuremath{ \mathcal{M} (\mathcal{#1})  } }
\newcommand{\Z}[1]{\mathbb{Z}/#1 \mathbb{Z} }
\newcommand{\Co}{ H^i_c(S,\mathbb{Z}_\ell ) } 
\newcommand{\Coh}{ H^i_c(S, \mathbb{Z} / \ell^n \mathbb{Z} ) } 
\newcommand{\Coho}[1]{ H^i_c(S,\mathbb{Z}/#1 \mathbb{Z} ) }
\newcommand{\Gal}{\text{Gal}(k^{sep} / k ) }  
\newcommand{\ibid}{\textit{ibid.} }

\newcommand{\Zl}{\ensuremath{\mathbb{Z}_\ell}}
\newcommand{\Ql}{\ensuremath{\mathbb{Q}_\ell}}
\newcommand{\CB}{\ensuremath{ \overset{L}{ \underset{ \mathbb{Z} / \ell^n \mathbb{Z} }{\otimes}    } } } 
 
\newcommand{\Zln}{\ensuremath{\mathbb{Z} / \ell^n \mathbb{Z}}}
\newcommand{\Zlnm}{ \ensuremath{ \mathbb{Z} / \ell^{n-1} \mathbb{Z}   }}
\newcommand{\HXR}{\ensuremath{  \mathbf{R} \Gamma_c((X,R) , \Zln )   }}
\newcommand{\HYT}{\ensuremath{  \mathbf{R} \Gamma_c((Y,T) , \Zln )   }}
\newcommand{\HXY}{\ensuremath{  \mathbf{R} \Gamma_c((X,R)\times(Y,T) , \Zln )   }}
\newcommand{\PT}{  \ensuremath{ \underset{\Zln}{\otimes} }   }

\newcommand{\CBl}{  \ensuremath {  \underset{\mathbb{Z}_\ell}{ \overset{L}{\otimes}}  }  }
\newcommand{\Req}{ \ensuremath{ \stackrel{\sim}{\rightarrow}  }}
\newcommand{\Leq}{ \ensuremath{ \stackrel{\sim}{\leftarrow}  }}

\newcommand{\PTZlL}{  \ensuremath{ \overset{L}{\underset{\mathbb{Z}_\ell}{\otimes} }}   }
\newcommand{\PTQlL}{  \ensuremath{ \overset{L}{ \underset{\mathbb{Q}_\ell}{\otimes} } }  }

\newtheorem{theo}{Theorem}[section]
\newtheorem{rem}[theo]{Remark}
\newtheorem{defi}[theo]{Definition}
\newtheorem{prop}[theo]{Proposition}
\newtheorem{lemme}[theo]{Lemma}

\newtheorem*{theor}{Theorem}
\newtheorem*{propo}{Proposition}

\title{Cohomology of locally-closed semi-algebraic subsets}
\author{Florent Martin} 
\thanks{ The research leading to these results has received funding from the European Research Council under the European Community's Seventh Framework Programme (FP7/2007-2013) / ERC Grant Agreement n$^o$ 246903 "NMNAG}
\address{Institut de Math\'{e}matiques de Jussieu\\
Universit\'{e} Pierre-et-Marie-Curie\\
4 Place de Jussieu, 75005 Paris}
\email{fmartin@math.jussieu.fr}
\begin{document}
\date{\today}
\setcounter{tocdepth}{1} 
\maketitle

\begin{abstract}
Let $k$ be a non-Archimedean field, let $\ell$ be a prime number 
distinct from the characteristic of the residue field of $k$.
If $\X$ is a separated $k$-scheme of finite type, 
Berkovich's theory of germs allows to define \'etale $\ell$-adic cohomology groups with compact support of locally closed semi-algebraic subsets of $\X^{an}$. 
We prove that these vector spaces are finite dimensional continuous representations of the 
Galois group of $k^{sep}/k$, and satisfy the usual long exact sequence and K\"unneth formula. This has been recently used by E. Hrushovski and F. Loeser
in a paper about the monodromy of the Milnor fibration. 
In this statement, the main difficulty is the finiteness result, whose 
proof relies on a cohomological finiteness result for affinoid spaces, recently proved by V. Berkovich.
\end{abstract}

\tableofcontents

\section{Introduction}
Let $k$ be a non-Archimedean\footnote{This means that $k$ is equipped with an 
ultrametric norm $| \cdot | : k \to \R_+$ for which it is complete.}
field and $\mathcal{X}$ 
a separated $k$-scheme of finite type. One can associate to it 
a $k$-analytic space $\mathcal{X}^{an}$ \cite{Berko90}. 
Using \cite{Berko93} one 
can define $\ell$-adic cohomology groups 
$H^i_c(\mathcal{X}^{an} , \mathbb{Q}_\ell )$ which have good properties if 
$\ell$ is different from char$(\tilde{k})$ (in particular, they are finite dimensional vector 
spaces when $k$ is algebraically closed). \par 
If the $k$-scheme $\mathcal{X} =$Spec$(A)$ is affine, a subset $U$ of $\mathcal{X}^{an}$ is called 
\emph{semi-algebraic} if it is a finite Boolean combination of subsets 
of the form $\{ x\in \mathcal{X}^{an} \ \big| \ |f(x)| \leq \lambda |g(x)| \}$ 
where $f$ and $g$ belong to $A$ and $\lambda$ is a positive real number. 
This definition of semi-algebraic subsets extends to general $k$-varieties (see Definition \ref{defisagen}). 
If $U$ is a semi-algebraic subset of 
$\mathcal{X}^{an}$,
using the theory of $k$-germs developed in 
\cite{Berko93}, it is possible to define 
cohomology groups of the $k$-germ $(\X^{an},U)$, that we will denote by $H^i_c(U,\mathbb{Q}_\ell)$. We want to point out that in general, $U$ is not equipped with a structure of 
$k$-analytic space.
\par 
In this paper, we generalize the finiteness property mentioned above to  
locally closed semi-algebraic subsets of $\mathcal{X}^{an}$.
More precisely, let $\widehat{k^a}$ be the completion of the algebraic closure of 
$k$ and let us set $\overline{\X} := \mathcal{X}^{an} \widehat{\otimes}_k \widehat{k^a}$ and 
$\pi :\overline{\X}   \to \mathcal{X}^{an}$ the natural morphism. If 
$U$ is a subset of $\X^{an}$, we set 
$\overline{U} := \pi^{-1} (U)$. Our main result is then: 

\begin{theor} \textbf{\ref{theoprinc}}
Let $\mathcal{X}$ be a separated 
$k$-scheme of finite type of dimension $d$, $U$ a locally closed semi-algebraic subset of 
$\mathcal{X}^{an}$, and $\ell \neq $char$(\tilde{k})$ be a prime number. 
\begin{enumerate}
 \item The groups 
$H^i_c( \overline{U} , \mathbb{Q}_\ell) $ are finite dimensional $\mathbb{Q}_\ell$-vector spaces, 
on which $\Gal$ acts continuously, and $H^i_c( \overline{U} , \mathbb{Q}_\ell) =0$ for 
$i>2d$.
\item  If $V\subset U$ is a semi-algebraic subset which is open in $U$ and $F=U \setminus V$, then there is a long exact sequence of Galois modules 
\[\xymatrix{
\ar@{.>}[r] & H^i_c(\overline{V} , \mathbb{Q}_\ell ) \ar[r] &
H^i_c(\overline{U} , \mathbb{Q}_\ell ) \ar[r] & H^i_c(\overline{F} , \mathbb{Q}_\ell ) 
\ar[r] & H^{i+1}_c(\overline{V} , \mathbb{Q}_\ell ) \ar@{.>}[r] &
} \]
\item 
For all integer $n$ there are canonical isomorphisms of Galois modules: 
\[ 
\bigoplus_{i+j=n} H^i_c \left(\overline{U} , \mathbb{Q}_\ell \right) \otimes H^j_c \left( \overline{V}, \mathbb{Q}_\ell \right) 
\simeq H^n_c \left( \overline{U\times V} , \mathbb{Q}_\ell \right).
\]
\end{enumerate}
\end{theor}
We prove more generally this result when 
$\X$ is a separated $\A$-scheme of finite type where $\A$ is a $k$-affinoid algebra. The above result corresponds to the case 
$\A=k$. \par
This question was raised by 
F. Loeser and used in \cite{HrLo} where they study the Milnor fibration associated 
to a morphism $f: X \to \mathbb{A}^1_{\mathbb{C}}$ where $X$ is a smooth complex 
algebraic variety. The non-Archimedean field is then $k = \mathbb{C}((t))$.\par
In fact, the main point in the above theorem is to prove that when $k$ is algebraically closed, the groups 
$H^i_c(\overline{U}, \Z{\ell^n})$ are finite.
This is obtained as a consequence of another analogous result which does not involve algebraic objects. If $X$ is a $k$-affinoid space whose affinoid algebra is $\mathcal{A}$, we say that a subset $S$ of $X$ is semianalytic if it is a finite Boolean combination of subsets of the form $\{ x\in X \ \big| \ |f(x)| \leq \lambda |g(x)| \}$ 
where $f$ and $g$ belong to $\mathcal{A}$ and $\lambda>0$. 
Now if $X$ is a compact $k$-analytic space, we say that 
a subset $S\subset X$ is \emph{$G$-semianalytic} if there exists 
a finite covering $\{X_i\}$ of $X$  
by some affinoid domains such that for each $i$, 
$S\cap X_i$ is semianalytic in $X_i$. 
We then prove:
\begin{propo} \textbf{\ref{propber}}
Let us assume that $k$ is algebraically closed, and let $X$ be a compact $k$-analytic space. Then for any locally 
closed $G$-semianalytic subset $S$ of $X$ and $\Lambda$ a finite group whose cardinal is prime 
to char$(\tilde{k})$, the groups 
$H^q_c( (X,S), \Lambda )$ are finite. 
\end{propo}
We want to point out that this result relies deeply on the cohomological finiteness 
of affinoid spaces which has been recently\footnote{In a previous version of this work (arXiv:1210.4521, [v1]), the results of \cite{Berko13} were not available. 
Instead, we used \cite[Corollary 5.5]{Berko94} which is a cohomological finiteness result for affinoid spaces which are algebraizable in some sense. 
As a consequence, in this previous version, we obtained less general results. 
In particular, we could not prove proposition \ref{finitude}.} proved by V. Berkovich in \cite{Berko13}.\par 
We also want to mention that in the author's thesis (section 2.4) it is proved that our finiteness result proposition \ref{finitude} is also true 
if we assume that $S$ is an overconvergent subanalytic subset\footnote{If $X$ is a $k$-affinoid space, 
the class of overconvergent subanalytic subsets of $X$, introduced by H. Schoutens in \cite{Sch_sub}, properly contains the class of semianalytic subsets of $X$.} of $X$.
\par
Finally in section \ref{huber}, we explain some counterparts of these finiteness results 
for Huber's adic spaces.

\textbf{Acknowledgements} I would like to express my deep gratitude to 
J.F. Dat and A. Ducros for their encouragements and advices. In particular I would like to thank A. Ducros who suggested me to work on the question that F. Loeser had asked to him, 
and to use $k$-germs. 

\subsection*{Notations}
In what follows, $k$ will be a complete non-Archimedean field, $X$ a Hausdorff 
$k$-analytic space, and $S\subset X$ will always be a locally closed subset of $X$.\par 
\subsubsection*{The \'etale site of a germ} (see \cite[3.4]{Berko93})
If $S$ is a subset of $X$, 
$(X,S)$ is called a $k$-germ. 
If $(Y,T)$ is another $k$-germ, a morphism of 
$k$-germs $f:(Y,T)  \to (X,S)$ is a morphism of 
$k$-analytic spaces $f:Y \to X$ such that 
$f(T)\subset S$. This defines the category of $k$-germs. \par  
Then the category of $k$-$\mathscr{G}$erms is defined as the localization of the category 
of $k$-germs by the morphisms of 
$k$-germs $\varphi : (Y,T) \to (X,S)$ which induce an isomorphism of 
$Y$ with some open neighbourhood of $S$ in $X$ (this implies that 
$\varphi$ induces a homeomorphism between $T$ and $S$). 
It is important to remark that 
the functor $k$-$\mathscr{A}$n $\to k$-$\mathscr{G}$erms defined by 
$X \to (X, |X|)$ is fully faithful. \par 
A morphism of $k$-$\mathscr{G}$erms $f$ is called \'etale if it has a representative 
$\varphi:(Y,T)\to(X,S)$ such that $\varphi : Y \to X$ is \'etale and 
$T=\varphi^{-1}(S)$.
Berkovich defines the small \'etale site $(X,S)_{\text{\'et}}$ of the $k$-germ $(X,S)$, 
as the category 
of \'etale morphisms above $(X,S)$, a family 
$(Y_i,T_i) \xrightarrow[]{f_i} (Y,T)$ being a covering if 
$\cup_if_i(T_i)=T$.
The category of abelian sheaves on  $(X,S)_{\text{\'et}}$ 
is denoted by $\boldsymbol{S}(X,S)_{\text{\'et}}$.
\subsubsection*{Cohomology groups with compact support} (see \cite[5.1]{Berko93})
If $S$ is a topological space, a family of supports 
$\Phi$ is a family of closed subsets of $S$ 
which is stable under finite unions and such that if $F$ is a closed subset of $S$ and 
$F \subset T$ for some $T \in \Phi$  then $F\in \Phi$. 
If $\Phi $ is a family of supports of $S$, and  $A \subset S$ we set
\begin{equation} 
\label{support}
\Phi_A = \{F\in \Phi \ \big| \ F\subset A \}
\end{equation}
which is a family of supports of  $A$.\par 
A family of supports $\Phi$ is said to be paracompactifying if for all $A\in \Phi$, $A$ is paracompact, and
if for all $A\in \Phi$, there exists  $B\in \Phi$ which is a neighbourhood of $A$. \par 
If $\Phi$ is a family of supports, the following functor 
\begin{equation}
\label{cohosupp}
\begin{array}{crcl}
\Gamma_\Phi : & \boldsymbol{S}(X,S)_{\text{\'et}} & \rightarrow & \text{Ab} \\
           & F & \mapsto & \{s \in F(X,S) \ \big| \ \text{supp}(s) \in \Phi \} 
\end{array}
\end{equation}
is left exact. Its right derived functors are denoted by 
$H^q_\Phi ((X,S),F)$. \par 
Let us assume that $S$ is locally closed in  $X$, that $T$ is an open subset of $S$, 
and that $R:=S \setminus  T$. If $\Phi$ is a paracompactifying family of 
supports of $S$, and 
$F$ is an abelian sheaf on $(X,S)$, there is a 
long exact sequence \cite[5.2.6 (ii)]{Berko93}:
\begin{align}
\label{longexacts}
\cdots \to H^{q-1}_{\Phi_R} ( (X,R) , F_{(X,R)} ) 
 \to H^q_{ \Phi_T} ((X,T),F_{(X,T)}) \to H^q_{\Phi} ( (X,S) ,F) 
 \to H^q_{\Phi_R} ((X,R) , F_{(X,R)} ) \to \cdots . 
\end{align} \par
Now, we denote by $C_S$ the family of compact subsets of $S$. 
If $S$ is Hausdorff, this is a family of supports, and if 
$S$ is locally compact, $C_S$ is paracompactifying. 
Remind that if $S$ is a locally closed subset of a locally compact 
topological set $X$, $S$ is also locally compact. 
Since we will always consider locally closed subsets $S$ of some Hausdorff 
$k$-analytic space $X$, the family of supports $C_S$ will then be paracompactifying.  
If $F \in \boldsymbol{S}(X,S)_{\text{\'et}}$,  
we will denote by 
\[H^q_c( (X,S), F):= H^q_{C_S} ((X,S),F)\] 
the associated right derived functors. \par 
Let $T\subset S$ be an open subset of $S$ and 
$R:=S\setminus T$ the complementary closed subset of $S$.
Then, ${(C_S)}_T =C_T$ because being compact in  $T$ or in $S$ is equivalent; 
likewise ${(C_S)}_R = C_R$. 
In this context, the long exact sequence \eqref{longexacts} 
can be written: 
\begin{align}
\label{sel1}
\cdots \to H^{q-1}_c((X,R),F_{(X,R)}) \to  H^q_c((X,T), F_{(X,T)}) \to 
H^q_c((X,S),F) \to H^q_c((X,R),F_{(X,R)})\to\cdots 
\end{align}

\subsubsection*{What we will look at}
Let $\Lambda$ be a finite abelian group whose cardinal is prime to the characteristic of $\tilde{k}$. 
We set
\begin{equation}
\label{coh}
H^n_c(S,\Lambda ) := H^n_c ((X,S),\underline{\Lambda})
\end{equation} 
where
$\underline{\Lambda}$ is the constant sheaf of value 
$\Lambda$ on  $(X,S)_{\acute{e}t}$. This notation is abusive since the cohomology 
of $S$ itself is meaningless, only the cohomology 
of the $k$-germ $(X,S)$ can be defined. 
Nonetheless, we will use the notation $\eqref{coh}$ to simplify the exposition. \par 
If we still denote by 
$\underline{\Lambda}$ the constant sheaf of value $\Lambda$ on $(X,S)_{\acute{e}t}$, then if $U \subset S$, 
$\underline{\Lambda}_{(X,U)}$ is isomorphic to the constant sheaf of value 
$\Lambda$ on $(X,U)_{\acute{e}t}$. 
Hence, if 
$T$ is an open subset of $S$ and $R:=S \setminus T$, 
the long exact sequence \eqref{sel1} becomes 
\begin{equation}
 \label{sel}
 \cdots \to H^{q-1}_c(R,\Lambda) \to  H^q_c(T,\Lambda) \to 
H^q_c(S,\Lambda) \to 
H^q_c(R,\Lambda) \to  \cdots 
\end{equation}

\subsubsection*{Quasi-immersions (see \cite[4.3]{Berko93})}
A morphism of $k$-germs 
$\varphi : (Y,T) \to (X,S)$ is called a quasi-immersion if 
$\varphi$ induces a homeomorphism of $T$ with its image 
$\varphi(T)$ and for all $y\in T$, if we set 
$x:=\varphi(y)$, the maximal purely inseparable 
extension 
of $\mathcal{H}(x)$ in $\mathcal{H}(y)$ is everywhere dense in $\mathcal{H}(y)$.\par
Here are two examples of quasi-immersions that we will use frequently. 
If $U$ is an analytic domain of $X$, 
and $\varphi : U \to X$ is the natural inclusion morphism, then 
$\varphi: (U,U) \to (X,U)$ is a quasi-immersion. 
If $\varphi : Z \to X$ is a closed immersion, then 
$(Z,Z) \to (X,Z)$ is a quasi-immersion. 
Moreover, quasi-immersions are stable under composition and base change.\par 
Quasi-immersions will be very important for us through the following result: 
\begin{propo} \cite[4.3.4 (i)]{Berko93}
If $\varphi:(Y,T) \to (X,S)$ is a quasi-immersion of 
$k$-germs, it induces an equivalence of categories 
 \[ \boldsymbol{S}(Y,T)_{\text{\'et}} \simeq  \boldsymbol{S}(X,\varphi(T))_{\text{\'et}}.\]
\end{propo}
In particular if $U$ is an analytic domain of $X$, there are isomorphisms
\begin{equation}
\label{invdom}
H^q_c(U, \Lambda) \simeq H^q_c( (U,U) ,\Lambda) \simeq H^q_c( (X,U) , \Lambda). 
\end{equation} 
Similarly, if $Z$ is a closed $k$-analytic subset of 
$X$, 
\begin{equation}
\label{invZar}
H^q_c(Z, \Lambda)\simeq H^q_c( (Z,Z), \Lambda) \simeq H^q_c( (X,Z) , \Lambda).
\end{equation} 
We want to stress out that \eqref{invdom} and \eqref{invZar} partly justify the abuse of the notation made in \eqref{coh}.

\section{Semianalytic and semi-algebraic sets}

\begin{defi}
\label{defi1}
Let $\A$ be a $k$-affinoid algebra, and let $X=\affin{A}$ be the associated $k$-affinoid space.
A subset $S \subset X $ is called \emph{semianalytic} if it is a finite Boolean combination\footnote{Since we authorize finite Boolean combinations, we could have also authorized some $<$ and $=$ in this definition.} of sets of the form
\begin{equation}
\label{defsa}
\{x\in X  \ \big|  \ |f(x)| \leq \lambda |g(x)| \} 
\end{equation}
where $f$, $g\in \mathcal{A}$, 
and $\lambda$ is a positive real number.
\end{defi}
If one takes $g=0$ in \eqref{defsa}, one sees that 
if $f\in \mathcal{A}$, the hypersurface 
$V(f) = \{x\in X \st f(x)=0\}$ is semianalytic. 
More generally, any Zariski-closed subset of $X$ is semianalytic. 
Using the Gerritzen-Grauert theorem, one can also 
check that any affinoid domain of $X$ is semianalytic.
\begin{defi}
\label{defsan}
Let $X$ be a compact $k$-analytic space. 
A subset $S \subseteq X$ is called \emph{G-semianalytic} 
if there exists  a finite cover 
$X= \cup_{i=1}^n U_i$ by affinoid domains such that 
$S\cap U_i$ is semianalytic in $U_i$ for all $i$.
\end{defi}
\begin{rem}
Let $X$ be a $k$-affinoid space. If $S$ is a semianalytic subset of $X$, then it is also 
a $G$-semianalytic subset of $X$ (just consider the trivial cover with $n=1$ and 
$X_1=X$). The converse is false: if 
$S\subset X$ is $G$-semianalytic, it is not necessarily semianalytic in $X$. \par
We could have said that $S$ is locally semianalytic if for each point 
$x\in X$, there is some affinoid neighbourhood $V$ of $x$ such that 
$S\cap V$ is semianalytic in $V$. With this definition, for a subset 
$S$ of a $k$-affinoid space $X$, the following implications hold
\[ S \ \text{is semianalytic} \ \Rightarrow 
S \ \text{is locally semianalytic} \ \Rightarrow \ S \ \text{is G-semianalytic}.\] 
However, these three classes of subsets are pairwise distinct. For more details on this, we refer to 
\cite{LR2001} and \cite{MarOve}.
\end{rem}
We remind the following definition of \cite[2.1]{Duc_sa}:

\begin{defi}
\label{defisemalg}
Let $\mathcal{A}$ be a $k$-affinoid algebra, 
$\mathcal{B}$ an $\mathcal{A}$-algebra of finite type, and 
$\X = \Spec{\mathcal{B}}$. A subset $S\subset \X^{an}$ is called 
\emph{semi-algebraic} if 
it is a Boolean combination of subsets
\[ \big\{ x\in \X^{an} \st \ |f(x)| \leq \lambda |g(x)| \big\} \]
where $f,g\in \mathcal{B}$ and $\lambda \in \R_+^*$.
\end{defi}

\begin{rem}
We want to point out that this definition depends on the algebraic datum 
$\X = \Spec{ \mathcal{B}}$ and not only on the $k$-analytic space $\X^{an}$.  \par
For instance, let us consider the case $\A =k$ and 
$\X = \Spec{ k[T_1,T_2]}$, so that 
\begin{equation}
\label{pres}
\X^{an} \simeq \mathbb{A}^{2,an}_k.
\end{equation}
Let us consider $S=\{x\in \mathbb{A}^{2,an}_k \st T_2(x)=0\}$. Then $S$ is semi-algebraic in $\mathbb{A}^{2,an}_k$ with respect to the presentation \eqref{pres}. 
Let then $f=\sum_{n\geq 0} a_nT_1^n$ be a series with $a_n\in k$ whose 
radius of convergence is infinite, and such that $f$ is not a polynomial. 
Let us then consider the automorphism of $\mathbb{A}^{2,an}_k$  defined by 
\[ \begin{array}{clcl}
\varphi : & \mathbb{A}^{2,an}_k & \to & \mathbb{A}^{2,an}_k \\
           & (T_1,T_2) & \mapsto & (T_1,T_2+f(T_1)

\end{array}
\]
It is easy to check that 
$\varphi(S) \subset \mathbb{A}^{2,an}_k$ is not semi-algebraic any more with respect to the presentation \eqref{pres}. \par 
So, in order to be precise, when one talks about a semi-algebraic subset $S \subset \X^{an}$, one should always specify the algebraic datum $\X$. However, for simplicity, we will not do it when the algebraic presentation will be clear from the context.
\end{rem}

\begin{lemme}
\label{lemmesa1}
Let $\A$ be a $k$-affinoid algebra, $\X$ an affine $\A$-scheme of finite 
type, and let $\X = \U_1 \cup \ldots \cup \U_n$ be an affine covering of $\X$. 
If $S\subset \X^{an}$, then $S$ is semi-algebraic in $\X^{an}$ 
if and only if $S \cap \U_i^{an}$ is semi-algebraic in 
$\U_i^{an}$ for all $i$.
\end{lemme}

\begin{proof}
Let $S\subset \X^{an}$. \par
In one hand it follows from definition~\ref{defisemalg} that if $S$ is semi-algebraic in 
$\X^{an}$, 
then $S\cap \U_i^{an}$ is also semi-algebraic in $\U_i^{an}$, using 
the restriction morphisms 
$\mathcal{O}_\X(\X) \to \mathcal{O}_\X(\U_i)$. \par 
Conversely, if $S\cap \U_i^{an}$ is semi-algebraic in $\U_i^{an}$ for all $i$, 
according to 
\cite[2.6]{Duc_sa}, $S \cap \U_i^{an}$ is also semi-algebraic in $\X^{an}$. 
Since $S = \cup_{i=1}^n (S \cap \U_i^{an} )$, 
it follows that $S$ is semi-algebraic in $\X^{an}$.     
\end{proof}

\begin{lemme}
\label{lemmesa2}
Let $\X$ be a separated $\A$-scheme of finite type, and 
$S\subset \X^{an}$. Let  $\X = \U_1 \cup \ldots \cup \U_n$ 
be some affine covering of $\X$. The following statements are equivalent:
\begin{enumerate}
\item 
For $i=1\ldots n$, $S\cap \U_i^{an}$ is semi-algebraic in $\U_i^{an}$.
\item For every open affine subscheme 
$\V \subset \X$, $S\cap \V^{an}$ is semi-algebraic in $\V^{an}$. 
\end{enumerate}
\end{lemme}

\begin{proof}
Let us assume that the condition $(1)$ is fulfilled. 
Let $\V$ be some open affine subscheme of $\X$. Then 
$\U_i\cap \V$ is an open affine subscheme of $\U_i$ (because $\X$ is separated), and 
since $S\cap \U_i^{an}$ is semi-algebraic in $\U_i^{an}$, 
$S\cap (\U_i \cap \V)^{an}$ is semi-algebraic in $(\U_i \cap \V)^{an}$. 
Since the family $\{\U_i \cap \V \}$ is a finite affine covering of 
$\V$ and since 
$S \cap \V^{an} = \cup_{i=1}^n S\cap(\U_i \cap \V)^{an}$, thanks to the previous lemma, 
$S \cap \V^{an}$ is semi-algebraic in $\V^{an}$, that is to say $(2)$ is true.\par
For the converse statement one just has to take $\V= \U_i$. 
\end{proof}

\begin{defi}
\label{defisagen}
Let $\mathcal{A}$ be a $k$-affinoid algebra, 
and let $\X$ be a separated $\A$-scheme of finite type. 
A subset $S\subset \X^{an}$ is called semi-algebraic if 
it satisfies one of the two equivalent conditions of lemma \ref{lemmesa2}.
\end{defi}

\begin{rem}
\label{remindcouv}
Thanks to lemma \ref{lemmesa2}, we can check whether $S\subset \X^{an}$ is semi-algebraic, 
with an affine covering of $\X$, and this does not depend on the covering. 
Moreover, if $\X$ is affine, according to lemma \ref{lemmesa1}, 
the two definitions \ref{defisemalg} and \ref{defisagen} of a semi-algebraic set are equivalent.
\end{rem}

We have already used the following result, which is proved as a consequence of the 
quantifier elimination in ACVF:
\begin{propo} \cite[2.5]{Duc_sa}
Let $\mathcal{A}$ be a $k$-affinoid algebra, $\X$ and $\Y$ some affine 
$\A$-schemes of finite type, $f:\X \to \Y$ an $\A$-morphism of finite type, 
and $S$ a semi-algebraic subset of $\X^{an}$. Then $f^{an}(S)$ is a semialgebraic subset 
of $\Y^{an}$.
\end{propo}

Thanks to lemma \ref{lemmesa1} and \ref{lemmesa2}, it has the immediate generalization: 
\begin{prop} 
\label{propstabsa}
Let $\mathcal{A}$ be a $k$-affinoid algebra, $\X$ and $\Y$ some separated 
$\A$-schemes of finite type, $f:\X \to \Y$ an $\A$-morphism of finite type, 
and $S$ a semi-algebraic subset of $\X^{an}$. Then $f^{an}(S)$ is a semialgebraic subset 
of $\Y^{an}$.
\end{prop}

\begin{rem}
We want to point out that definition \ref{defisagen} both generalizes the definition of a semianalytic set of an affinoid space, and the more classical definition of a 
semi-algebraic set of some $\X^{an}$ where 
$\X$ is an affine $k$-scheme of finite type.
\end{rem}

\begin{rem}
\label{remaffi}
Let $\A$ be a $k$-affinoid algebra. 
Let $\X$ be an affine $\mathcal{A}$-scheme of finite type, $V$ be an 
affinoid domain of $\X^{an}$, and let
$S \subset \X^{an}$ be a semi-algebraic set.  
Then
$S\cap V$ is semianalytic in $V$. 
This is a straightforward consequence of the above definitions.
\end{rem}
\begin{lemme}
\label{lemmesemalgaffi}
Let $\mathcal{X}$ be a separated $\mathcal{A}$-scheme of finite type, 
$S \subset \X^{an}$ a semi-algebraic subset of 
$\X^{an}$ and 
$V$ an affinoid domain of $\X^{an}$. Then 
$S\cap V$ is $G$-semianalytic in $V$.
\end{lemme}
\begin{proof}
It is possible to find a finite covering of $V$ by affinoid domains 
$V= \cup_{i=1}^n V_i$, and some affine open subschemes 
$\U_i$ of $\X$ such that 
$V_i \subset \U_i^{an}$ for all $i$. 
Then $S\cap V_i   =   (S\cap \U_i^{an})\cap V_i$, and since 
$S \cap \U_i^{an}$ is semi-algebraic in $\U_i^{an}$ and 
$V_i \subset \U_i^{an}$, $S\cap V_i$ is semianalytic in $V_i$ according to remark \ref{remaffi}.
\end{proof}
In fact, in the above lemma, one can check that $S\cap V$ is even locally semianalytic.

\section{A finiteness result in the affinoid case}
\label{secfinaff}
In this section $k$ will be a (complete) non-Archimedean algebraically closed field. 
We consider  a $k$-affinoid algebra $\mathcal{A}$, and 
we set $X = \affin{A}$.
We remind that $\Lambda$ is a finite abelian group whose order is prime to the characteristic of $\tilde{k}$. 
The goal of this section is to prove proposition \ref{finitude}.
\begin{lemme}
\label{lemme1}
Let $n\in \N$,  and for $i=1 \ldots n$, let $f_i,g_i \in \mathcal{A}$,    
$\Diamond_i \in \{ < , \leq \}$ and $\lambda_i > 0$ be a positive real 
number. 
Let us consider   
\[S= \bigcap_{i=1}^n \{ x\in X \ \big| \ |f_i(x)| \Diamond_i \lambda_i |g_i(x)| \}.\]
Then, the groups $H_c^q(S,\Lambda ) $ are finite for all $q\in \mathbb{N}$.
\end{lemme}

\begin{proof}
We prove the lemma by induction on $n$.\par
If $n=0$, then  $S=X$ and the result is a consequence of the 
finiteness result  \cite[Theorem 1.1.1]{Berko13}. \par
Let $n\geq 0$ and assume that the result is true for $n$. 
Let $f,g \in \mathcal{A}$, $\lambda >0$, and let  
\begin{align*}
S= & \bigcap_{i=1}^n \Big\{ x\in X \ \big| \ |f_i(x)| \Diamond_i \lambda_i|g_i(x)|, \ i=1\ldots n \Big\} \\  
T=&\{x\in S \ \big| \ |g(x)| \lambda < |f(x) | \} \\
R =& \{x\in S  \ \big| \ |f(x)| \leq \lambda |g(x)| \} = S\setminus T.
\end{align*}
Let us show that the groups 
$H^q_c(T, \Lambda)$  and $H^q_c(R , \Lambda)$ are finite. This will achieve our induction step.\par 

By its definition, $S$ is a locally closed subset of $X$, $T$ is an
open subset of $S$, and $R = S \setminus T$ is the complementary closed subset of $S$.  
So we can apply the long exact sequence  \eqref{sel} 
to $S$, $R$ and $T$. By induction hypothesis, the groups 
$H^q_c(S, \Lambda)$ are finite, so if
we show that the groups $H^q_c(R , \Lambda)$ are finite, this will also prove 
the finiteness of the groups $H^q_c(T, \Lambda)$. Let us then show that the 
groups $H^q_c(R , \Lambda)$ are finite.\par
Let $Y = \mathcal{M} \left( \mathcal{A}\{\lambda^{-1}U\}/(f-Ug ) \right)$
and let $\varphi : Y \to X $ be the morphism of affinoid spaces induced by the natural map 
$\mathcal{A} \to \mathcal{A}\{\lambda^{-1}U\}/(f-Ug )$.  
The morphism $\varphi$ induces an isomorphism between the analytic domain of $Y$:
\[A = \{y \in Y \ \big| \ g(y) \neq 0 \} \] 
and the analytic domain of $X$:
\[B=\{x\in X \ \big| \ |f(x)| \leq \lambda |g(x)| \ \text{and} \ g(x)\neq 0 \}.\] 
As a consequence, $\varphi$ induces a quasi-immersion  
$\varphi : (Y,A) \to (X,B)$, and also a quasi-immersion 
\begin{equation}
\label{im}
(Y,A\cap \varphi^{-1} (S) ) \to (X,B \cap S).
\end{equation}

But 
\[\varphi^{-1}(S) = \bigcap_{i=1}^n \{y \in Y \ \big| \ |f_i(y)| \Diamond_i \lambda_i |g_i(y)| \}. \]
Here we  have written  $f_i$ (resp. $g_i$) whereas we should rather have written 
$\varphi^*(f_i)$ (resp. $\varphi^*(g_i)$).
Hence by induction hypothesis, the groups
$H^q_c((Y,\varphi^{-1}(S)) , \Lambda )$ are finite.\par
Now $A\cap \varphi^{-1}(S)$ is an open subset of
$\varphi^{-1}(S)$, whose complement  
in $\varphi^{-1}(S)$  is 
$\varphi^{-1}(S)\cap \{y\in Y \ \big| \ g(y) =0\}$.
Let  $Z$ be the Zariski closed subset of $Y$ defined by 
\[Z=\{y \in Y \ \big| \ g(y)=0 \}\] 
and
$\psi : Z \to Y $ the associated closed immersion. 
We then obtain a quasi-immersion:
\[(Z,\psi^{-1}( \varphi^{-1} (S) ) ) \to (Y, Z \cap \varphi^{-1}(S)).\]
By the induction hypothesis the groups
$H_c^q( \psi^{-1} ( \varphi^{-1} (S ) ) , \Lambda ) $ are finite, 
therefore it is also true for the groups 
$H^q_c (  Z \cap \varphi^{-1}(S), \Lambda ) $.
Thus in the long exact sequence 
\[\cdots \to  H^q_c(A\cap \varphi^{-1}(S) , \Lambda ) \to H^q_c (\varphi^{-1}(S) , \Lambda) \to H^q_c(Z \cap \varphi^{-1}(S) , \Lambda ) \to \cdots \]
the written groups in the middle and in the right are finite and we conclude from this that the groups
$H^q_c (A\cap\varphi^{-1}(S),\Lambda)$ are finite.
We have already noticed that $(Y,A\cap \varphi^{-1}(S))\to(X,B \cap S)$ is a 
quasi-immersion, hence 
\[H^q_c(A\cap\varphi^{-1}(S),\Lambda) \simeq H^q_c(B \cap S,\Lambda).\]
From this, we conclude that the groups  $H^q_c(B \cap S,\Lambda)$ are also finite.\par
If we go back to our starting point 
\[B\cap S= \{x\in S \ \big| \ |f(x)| \leq \lambda |g(x)| \ \text{and} \ g(x)\neq 0 \}  \]
is an open subset of 
\[R =  \{x\in S \ \big| \ |f(x)| \leq \lambda |g(x)| \}.\]
The complementary subset  of $B\cap S$ in  $R$ is  
\[D= \{x\in S \ \big| \ |f(x)|\leq \lambda |g(x)| \ \text{and} \ g(x) =0 \} 
= \{x\in S \ \big| \ f(x)=g(x) =0 \}.\]
We denote by  $Z'$ the Zariski closed subset of
$X$ : 
\[ Z' = \{x \in X \ \big| \  f(x)=g(x)=0 \}\] 
hence 
$D = Z' \cap S$, and using the same kind of arguments as above we can 
conclude that the groups 
$H_c^q(D,\Lambda)$ are finite.\par
We use for the last time the long exact sequence 
\[\cdots \to H^q_c(B\cap S , \Lambda) \to H^q_c(R , \Lambda) \to H^q_c(D,\Lambda ) \to \cdots \]
We have shown that the groups on the left, and on the right are finite, 
thus the groups
$H^q_c(R,\Lambda)$ are also finite.
\end{proof}

Next, we want to extend this result to an arbitrary locally closed 
semianalytic subset of $X$. In order to do so, we introduce the following notation. \par
Let  $f_1, \ldots f_r , g_1, \ldots g_r \in \mathcal{A}$, and 
$\lambda_1, \ldots , \lambda_r >0$. 
For a subset $I \subseteq \{1 \ldots r\}$ we set
\[C_I = \left( \bigcap_{i\in I} \{x\in X \ \big| \ |f_i(x)| \leq \lambda_i | g_i(x) | \} \right) \cap 
\left( \bigcap_{j\notin I }\{x\in X \ \big| \ |f_j(x)| > \lambda_j |g_j(x)| \} \right). \]
The subsets $C_I$ induce a partition of $X$, and each $C_I$ is a 
semianalytic set of $X$.
If $A \subseteq \mathcal{P}(\{1 \ldots r \})$, let us set
\[C_A = \coprod_{I \in A } C_I. \]
This is a semianalytic subset of $X$, and in fact every semianalytic subset of $X$ 
is of this form\footnote{This is some kind of disjunctive normal form.}. 
This follows from the fact that if $S$ is a semianalytic subset of $X$, one can find some 
$f_1, \ldots , f_r, g_1, \ldots , g_r \in \A$ such that $S$ is a finite union 
of subsets of the form 
\[\{x\in \ \big| \  |f_{i_1}(x)| \Diamond_{i_1} |g_{i_1}(x)| \ \text{and} \ 
\cdots  \ \text{and} \
|f_{i_m}(x)| \Diamond_{i_m} |g_{i_m}(x)| \} \]
where $1 \leq i_1 < \ldots < i_m \leq r$, and 
$\Diamond_j \in \{ \leq , > \}$. \par
For instance, 
if 
$S = \{|f_1| \leq |g_1| \} \cup \{|f_2| >|g_2| \}$, 
$A = \{ \{1,2\} , \{1\} , \emptyset \}$ is suitable:
\[S= \{|f_1| \leq |g_1| \ \text{and} \ |f_2| \leq |g_2| \} \cup 
 \{|f_1| \leq |g_1| \ \text{and} \ |f_2| > |g_2| \} \cup 
 \{|f_1| > |g_1| \ \text{and} \ |f_2| > |g_2| \}. \]

\begin{lemme}
\label{lemme21}
Let  $r$ and $n$ be two integers, $f_1, \ldots f_r , g_1, \ldots g_r, F_1, \ldots F_n , G_1, \ldots G_n\in \mathcal{A}$,  
$A\subseteq \mathcal{P}( \{1 \ldots r \} )$,  
$\Diamond_i \in \{ < , \leq \}$ for $i=1 \ldots n$
 and  $\lambda_1 , \ldots , \lambda_r , \mu_1 , \ldots , \mu_n $ be some positive real numbers.
Let us suppose that the semianalytic set of $X$
\[C= C_A \cap \Big(\bigcap_{i=1}^n \{x\in X \ \big| \ |F_i(x)| \Diamond_i \mu_i |G_i(x)| \} \Big)\]
is locally closed.
Then the groups $H^q_c(C,\Lambda )$ are finite. 
\end{lemme}

\begin{proof}
We prove this  by induction on $r$. \par
If $r=0$ this is precisely the preceding lemma \ref{lemme1}. \par
Let $r\geq 0$ and let us assume that  
we are given $f_1, \ldots f_{r+1}, g_1 , \ldots g_{r+1} \in \mathcal{A}$, 
$A\subset \mathcal{P} (\{1 \ldots r+1 \})$, and 
\[ C = C_A \cap \left( \bigcap_{i=1}^n \{x\in X \ \big| \ |F_i(x)| \Diamond_i \mu_i |G_i(x)| \} \right)\] 
a subset of $X$, assumed to be locally closed. 
Then we must show that the groups $H^q_c(C,\Lambda )$ are finite.
The idea is to decompose  $C$ as 
\[ C = \big\{ x\in C \st |f_{r+1}(x)| \leq \lambda_{r+1}|g_{r+1}(x)| \big\} 
\coprod  \big\{ x\in C \st |f_{r+1}(x)| > \lambda_{r+1}|g_{r+1}(x)| \big\} \] 
and to use our induction hypothesis to this partition of $C$. \par
To formalize this, we set
\[A_1 = \{P \in A \ \big| \ r+1 \in P \} \]
\[A_2= \{P \in A \ \big| \ r+1 \notin P \} = A \setminus A_1.\]
Finally we set
\[B_1 = \{P\setminus \{r+1 \} \ \big| \ P\in A_1 \} \]
and  $B_2 = A_2$. In addition, we see $B_1$ and $B_2$ as subsets of  
$\mathcal{P} ( \{1 \ldots r \} ) $.\par
We now consider the subsets of $X$, 
$C_{B_1}$  and $C_{B_2}$, associated with 
$f_1,\ldots f_r, g_1, \ldots g_r$ and $\lambda_1, \ldots , \lambda_r$ .
Then, by definition of $B_1$ and $B_2$, 
\[C_A =  ( \{x\in X \ \big| \ |f_{r+1} (x) | \leq \lambda_{r+1} | g_{r+1}(x)| \} \cap C_{B_1} )  \coprod
( \{x\in X \ \big| \ |f_{r+1} (x) | > \lambda_{r+1} | g_{r+1}(x)| \} \cap C_{B_2} ). \]
Said more simply, we have partitioned the set
$ C_A = \coprod_{I \in A } C_I$ in two parts: on the left side, we have kept the $C_I's$ where 
the inequality $|f_{r+1}(x)|\leq\lambda_{r+1}|g_{r+1}(x)|$ appears, and on 
the right side, we have kept the $C_I's$ where the inequality  
$|f_{r+1}(x)|>\lambda_{r+1}|g_{r+1}(x)|$ appears, which allows 
us to 
restrict to subsets of $\{1\ldots r \}$. 
And now we set: 
\begin{align*}
C_1 = C_{B_1} \cap \left( \{x\in X \ \big| \ |f_{r+1}(x) | \leq \lambda_{r+1} |g_{r+1} (x) |  \} \cap   
\bigcap_{i=1}^n  \{x\in X \ \big| \ |F_i(x)| \Diamond_i \mu_i |G_i(x)| \} \right) \  \\ 
 C_2=C_{B_2} \cap \left( \{x\in X \ \big| \ |f_{r+1}(x) | >\lambda_{r+1} |g_{r+1} (x)| 
\} \cap   
\bigcap_{i=1}^n  \{x\in X \ \big| \ |F_i(x)| \Diamond_i \mu_i |G_i(x)| \} \right).
 \end{align*} 
The following holds : 
\[ C_1 =\big\{ x\in C \st |f_{r+1}(x)| \leq \lambda_{r+1}|g_{r+1}(x)| \big\}  \]
\[C_2 =\big\{ x\in C \st |f_{r+1}(x)| > \lambda_{r+1}|g_{r+1}(x)| \big\} . \]
So $C= C_1 \coprod C_2$, $C_2$ is an open subset of $C$, and $C_1$ is 
the closed complementary subset attached to it, in particular, $C_1$ and $C_2$ are locally closed in $X$.
But we can now apply our induction hypothesis to  $C_1$ and $C_2$: 
the groups $H^q_c(C_i,\Lambda)$ are finite for $i=1,2$.  
Finally, according to long exact sequence \eqref{sel} applied to $C_2 \subset C \supset C_1$,  the groups $H^q_c(C,\Lambda ) $ are finite. 
\end{proof} 

The previous lemma, with $n=0$ becomes:
\begin{prop}
\label{finitude}
Let  $S$ be a  locally closed semianalytic subset of $X$.
The groups   
$H^q_c(S,\Lambda)$ are finite.
\end{prop}

\section{Global results}
In this section, we will still assume that $k$ is algebraically closed.
\subsection{The compact case}

\begin{prop}
\label{propber}
Let $X$ be a compact $k$-analytic space, and let $S$ be a locally closed 
G-semianalytic subset of $X$. 
The groups 
$H^q_{c}(S , \Lambda)$ are finite.
\end{prop}

\begin{proof}
We prove by induction on $m$ that if $X$ is a Hausdorff $k$-analytic space which is covered by 
$m$ affinoid domains: $X = \cup_{i=1}^m V_i$, and $S \subset X$ is a locally closed 
subset of $X$ such that 
$S \cap V_i$ is semianalytic in $V_i$ for all $i$ (in particular $S$ is 
$G$-semianalytic in $X$), then the groups
$H^q_{c} (S,\Lambda)$ are finite. \par
For $m=1$ this is proposition \ref{finitude}. \par
Let then $m\geq 1$ and let us assume that $X$ is covered by the affinoid domains 
$V_i, ~i=1\ldots m+1$, and that $S\subset X$ such that 
for all $i$, $S\cap V_i$ is semianalytic in $V_i$. We set 
\begin{align}
 \notag R &= S \cap ( V_1 \cup \ldots \cup V_m) \\
 \notag T &= S \setminus R \\
 \label{T2} & = (V_{m+1} \cap S) \setminus (V_{m+1} \cap(V_1 \cup \ldots \cup V_m)).
\end{align}
If $X':=V_1 \cup \ldots V_m$, then $X'$ is a compact analytic domain of $X$ (not necessarily 
good), and $R$ is a locally closed $G$-semianalytic subset of $X'$ such that 
$R \cap V_i$ is semianalytic in $V_i$ for $i=1 \ldots m$. 
Thus by induction hypothesis, the groups 
$H^q_c( (X',R), \Lambda)$ are finite. In addition, since 
$(X',R) \to (X,R)$ is a quasi-immersion (because $X'$ is an analytic domain of $X$), 
$H^q_c( (X',R), \Lambda) \simeq H^q_c( (X,R), \Lambda)$, hence these are finite groups.\par

Next, we claim that $T$ is a locally closed semianalytic set of $V_{m+1}$. 
Indeed, for each $i$, 
$V_{m+1} \cap V_i$ is an affinoid domain of $V_{m+1}$ 
(because $X$ is separated), hence closed and  semianalytic in 
$V_{m+1}$ according to the Gerritzen-Grauert theorem. Hence 
$V_{m+1} \cap ( V_1 \cup \ldots V_n)$ is a closed semianalytic set of $V_{m+1}$. 
Since $S \cap V_{m+1}$ is a locally closed semianalytic subset of $V_{m+1}$, according to 
\eqref{T2},  
$T$ is a locally closed semianalytic subset of $V_{m+1}$.
Hence according to proposition \ref{finitude}, the groups $H^q_c( (V_{m+1},T), \Lambda)$ are finite, 
and since 
$(V_{m+1},T) \to (X,T)$ is a quasi-immersion, 
$H^q_c( (V_{m+1},T), \Lambda) \simeq H^q_c( (X,T), \Lambda)$, thus 
the groups $H^q_c( (X,T), \Lambda)$ are finite.\par
Finally, since $R$ is a closed subset of $S$ and $T=S \setminus R$, 
the long exact sequence \eqref{sel}
allows to conclude that the groups $H^q_c( (X,S),\Lambda)$ are finite.
\end{proof}

\subsection{The semi-algebraic case}

\begin{prop}
\label{finisemialg}
Let $\mathcal{A}$ be a $k$-affinoid algebra, $\mathcal{X}$ a separated $\mathcal{A}$-scheme of finite type, and $S$ a locally closed semi-algebraic subset of 
$\mathcal{X}^{an}$. Then the groups 
$H^q_{c}((\X^{an},S) , \Lambda)$ (that we abusively denote by 
$H^q_c(S,\Lambda)$) are finite.
\end{prop}

\begin{proof}
According to Nagata's compactification theorem (see \cite{ConNag} for a modern proof), 
we can embed $\X$ as an open subscheme of 
a proper $\A$-scheme $\overline{\X}$. Since 
$(\X^{an},\X^{an}) \to (\overline{\X}^{an}, \X^{an})$ is a quasi-immersion, and since 
quasi-immersions are stable under base change, for all $q\geq 0$ we have an 
isomorphism of groups:
\[H^q_c( (\X^{an},S), \Lambda) \simeq  H^q_c( (\overline{\X}^{an},S), \Lambda).\]
In addition, according to proposition \ref{propstabsa}, $S$ is still 
semi-algebraic in $\overline{\X}^{an}$. Moreover, 
$S$ is still locally closed in $\overline{\X}^{an}$ because 
$\X^{an}$ is open in $\overline{\X}^{an}$. So we can assume that $\X$ is proper. \par 
In that case, $\X^{an}$ is compact, 
and according to lemma \ref{lemmesemalgaffi}, $S$ is $G$-semianalytic in $\overline{\X}^{an}$, and 
the result follows from proposition \ref{propber}.

\end{proof}

\section{From torsion to $\ell$-adic coefficients}

\subsection{Continuous Galois action}
\label{Galois}
From now on, we do not assume any more that $k$ 
is algebraically closed. We still consider 
$X$ a Hausdorff $k$-analytic space.
Let  $S$ be a  locally closed subset of  $X$ 
and let us set  $\overline{X} =X\widehat{\otimes}_k \widehat{k^a}$, 
$\pi:  \overline{X} \to X$, the projection, and  $\overline{S} = \pi^{-1} (S)$. 
This is a locally closed subset of $\overline{X}$.
There is an action of $\Gal$ on $\overline{X}$ which stabilizes $\overline{S}$.
Hence $\Gal$ acts on the $k$-germ  $(\overline{X} , \overline{S} ) $. 
If $\mathcal{F} \in \boldsymbol{S}(X,S)$, we set 
$\overline{\mathcal{F}} = \pi^*(\mathcal{F})$. 
The action of $\Gal$ on $(\overline{X} , \overline{S})$ induces an action on 
 $H_c^i( \overline{S} , \overline{\mathcal{F}} )$. 
Indeed for $\sigma \in \Gal$ we have the commutative diagram 
\[\xymatrix{
(\overline{X},\overline{S} ) \ar[rd]^{\pi}  \ar[rr]^{\sigma} & &(\overline{X},\overline{S} ) \ar[ld]^{\pi}  \\
& (X,S) & \\ 
} \]
Then the action of $\sigma$ on the cohomology is given by :
\[\sigma^* : \ H^i_c( (\overline{X},\overline{S}) , \overline{\mathcal{F}} ) \simeq 
H^i_c( (\overline{X},\overline{S}) , \sigma^* \overline{\mathcal{F}} )  \simeq 
H^i_c( (\overline{X},\overline{S}) , \overline{\mathcal{F}} ), \]
the last isomorphism being a consequence of the isomorphism 
$\sigma^* \circ \pi^ * ( \mathcal{F} ) \simeq \pi^ * ( \mathcal{F} )$.
If $(X,S)$ is a $k$-germ, and $K$ is a complete extension of  $k$, 
we consider  
$\pi_K: X_K =  X\hat{\otimes}_k K \to X $ and we set $S_K = \pi_K^{-1} (S)$, so that 
we can  consider the $K$-germ  $(X_K, S_K)$.

\begin{prop}
If $X$ is a Hausdorff $k$-analytic space, $F$ a locally closed subset of  $X$, 
$\mathcal{F} \in \boldsymbol{S}_{\text{\'et}}(X)$, there is an isomorphism of Galois modules:
$$\varinjlim_{K / k} H^q_c( (X_K , F_K) , \mathcal{F}_K ) \simeq 
H^q_c( (\overline{X} , \overline{F} ) , \overline{ \mathcal{F} } ) $$
where the limit is taken over all finite separable extensions $K$ of $k$ contained in $k^{sep}$.
\end{prop}
\begin{proof}
We will use that if $Y$ is a Hausdorff $k$-analytic space and $\mathcal{G} \in \boldsymbol{S}_{\text{\'et}}(Y)$,  
the following is true  \cite[5.3.5]{Berko93}:
\[\varinjlim_{K /k} H^q_c(Y_K , \mathcal{G}_K ) \simeq H^q_c (\overline{Y}  , \overline{ \mathcal{G} } ), \]
and it is an isomorphism of Galois-modules.\par
Since $F$ is locally closed, it can be written 
$F = U \cap F'$ where $U$ is open in $X$ and $F'$ is closed in $X$, 
and since $(U,F) \to (X,F)$ is a quasi-immersion, for all $q\geq 0$, 
$H^q_c( (U,F),\Lambda ) \simeq H^q_c( (X,F) ,\Lambda )$. 
Now,  
$F$ is closed in $U$, so we can replace $X$ by $U$ and assume that $F$ is closed. \par 
In this situation, let $U = X \setminus F$ be the complementary open subset of $X$.
For $K$ a finite separable extension of $k$, $F_K$ is a closed subset of $X_K$ whose complementary open subset is  $U_K$.
Hence we get a commutative diagram:
\[\xymatrix{
\ar@{.>}[r] & \varinjlim_{K /k} H^q_c(U_K,\mathcal{F}_K) \ar@{=}[d] \ar[r] &
\varinjlim_{K / k } H^q_c(X_K , \mathcal{F}_K) \ar@{=}[d] \ar[r] & 
\varinjlim_{K / k } H^q_c( (X_K,F_K) , \mathcal{F}_K) \ar[d] \ar@{.>}[r]   & \\
\ar@{.>}[r] &  H^q_c(\overline{U},\overline{\mathcal{F}})  \ar[r] &
 H^q_c(\overline{X}  , \overline{\mathcal{F}})  \ar[r] & 
 H^q_c((\overline{X}, \overline{F} ) , \overline{\mathcal{F}})  \ar@{.>}[r]   & 
}\]
Thanks to the long exact sequence \eqref{sel1}, 
the first row is exact because $\varinjlim$ is an exact functor (we consider a filtered inductive limit), and the second row is exact.   
We can then conclude thanks to the five lemma.\end{proof}
In particular, if $\A$ is a $k$-affinoid algebra, $\X$ is a separated  
$\A$-scheme of finite type,  
and if $S$ is a locally closed subset of $\mathcal{X}^{an}$,
\[H^q_c((\overline{\X},\overline{S}), \Lambda ) \]
is a continuous Galois module. Moreover, if $T$ is an open subset of $S$ 
and $R=S\setminus T$, the long exact sequence
\[ \cdots \to H^q_c((\overline{\X} , \overline{T}),\Lambda) \to 
H^q_c( (\overline{\X} , \overline{S}),\Lambda) \to 
H^q_c(  ( \overline{\X} , \overline{R}),\Lambda) \to  \cdots \]
is Galois equivariant.

\subsection{About the dimension}
Let $X$ be a Hausdorff $k$-analytic space. 
We denote by $d$ the dimension of $X$ (cf. \cite[p.~34]{Berko90} and \cite[p.~23]{Berko93}).
\begin{prop}
\cite[Cor 5.3.8]{Berko93}. 
Let $Y$ be a Hausdorff $k$-analytic space of  dimension $d$, $\mathcal{F}$ 
a torsion abelian sheaf on
$Y$, then for all  $i>2d$,   
$H_c^i(Y,\mathcal{F}) =0$.
\end{prop}
\noindent We can generalize this result in the following way:
\begin{prop}  
\label{prop_dim}
Let $X$ be a Hausdorff $k$-analytic space 
of dimension $d$. 
Let  $S$ be a locally closed subset of $X$.  For $q>2d$, 
and  $F \in \boldsymbol{S}(X)$ an abelian torsion sheaf on $X$, 
$H^q_c((X,S), F ) =0$.
\end{prop}
\begin{proof} 
Write $S=U\cap Z$ with $U$ an open subset of $X$ and  $Z$ a closed subset.
Set $V = U \setminus S$ which is an open subset of $U$ and $X$.
Then 
$H^q_c ((U,S),F) \simeq H^q_c ((X,S),F)$ and 
$H^q_c(V,F) \simeq H^q_c((V,V),F) \simeq  H^q_c((U,V),F)$, hence in the long exact sequence \eqref{sel} 
\[\cdots \to H^q_c( (U,V) , F) \to H^q_c(U,F) \to H^q_c((U,S) , F) \to \cdots  \]
according to the previous proposition, the groups are 
$0$ on the left and in the middle for 
$q>2d$, so this must also occur for the groups on the right.
\end{proof}
In our situation, this result can be refined. 
If $\A$ is a $k$-affinoid algebra, $\X$ a separated $\A$-scheme of finite type, and 
$S\subset \X^{an}$ a semi-algebraic set, we set 
$Z := \overline{S}^{Zar}$. 
Then 
since $(Z,S) \to (X,S)$ is a quasi-immersion, 
$H^q_c((Z,S), F) \simeq H^q_c((X,S),F)$. Hence if 
we set 
$\dim(S):= \dim(Z)$, with the above notations, $H^q_c((X,S), F ) =0$
 for all $q>2\dim(S)$.

\subsection{Finiteness of the $\ell$-adic cohomology}
In this subsection, we assume again that $k$ is algebraically closed. We fix $\A$ a $k$-affinoid 
algebra, $\X$ a separated $\A$-scheme of finite type, $S$ a locally closed 
semi-algebraic subset of $\X^{an}$, and $\ell$ a prime number different from the 
characteristic of $\tilde{k}$.\par  
In this situation, we have seen in proposition 
\ref{finisemialg} that for $n\geq 0$, the groups 
$H^q_c(S,\Z{\ell^n})$ are finite (we remind that the notation 
$H^q_c(S,\Z{\ell^n})$ is a shorthand for $H^q_c( (X,S),\underline{\Z{\ell^n}})$). \par 

We then set
\begin{equation}
\label{cohzl}
H^q_c(S, \mathbb{Z}_\ell ) =\varprojlim_{n>0} H^q_c(S, \mathbb{Z} / \ell^n \mathbb{Z} )   
\end{equation}
and 
\begin{equation}
\label{cohql}
H^q_c(S, \mathbb{Q}_\ell ) =H^q_c(S, \mathbb{Z}_\ell ) \otimes_{\mathbb{Z}_\ell} \mathbb{Q}_\ell . 
\end{equation}
It is a classical fact that proposition \ref{finisemialg} implies that the groups  
$H^q_c(S, \mathbb{Z}_\ell )$ are finitely generated $\mathbb{Z}_\ell$-modules, and as a 
consequence, that $H^q_c(S, \mathbb{Q}_\ell )$ are finite-dimensional $\mathbb{Q}_\ell$-vector spaces. For completeness, we give here a proof as simple as possible. 
\begin{defi}
A $\mathbb{Z}_\ell$-module $M$ is called complete and separated 
(with respect to the $\ell$-adic topology) if the canonical map 
\[\pi : M \to \widehat{M} := \varprojlim_{k\geq 1} M / \ell^kM \]
is an isomorphism.
\end{defi}


\begin{prop} 
\label{propJ}
Let us consider a projective system of abelian groups
\[M_1 \xleftarrow[]{d_1} M_2 \leftarrow \cdots \xleftarrow[]{d_{n-1}} M_n \xleftarrow[]{d_n}  \cdots \]
where each $M_n$ is a finite $\Z{\ell^n}$-module. Then 
\[M := \varprojlim_{n\geq 1} M_n \]
is a complete and separated $\Zl$-module.
\end{prop}

\begin{proof}
We must show that 
\[\pi : M \to \widehat{M} = \varprojlim_{k\geq 1} M / \ell^kM \]
is an isomorphism.\par
We first prove that $\pi$ is injective. 
If $x = (x_n) \in M$, and 
$\pi(x) =0$, this means that $x \in \ell^kM$ for all $k$. 
Since for each $n$, 
$M_n$ is a $\Z{\ell^n}$-module, taking $k=n$, this implies that $x_n \in \ell^nM_n$, so $x_n=0$ for all $n$, hence $x=0$. \par
Let us now prove that $\pi$ is surjective. 
For this, we consider a Cauchy sequence $(y^{(k)})_{k\geq 1}$ in $M$, such that 
for all $j\geq k$, $y^{(j)} \equiv y^{(k)} \mod \ell^kM$. 
In particular, if $j\geq k$, this implies that for all $n$,
\begin{equation}
\label{congruence}
 y^{(j)}_n \equiv y^{(k)}_n \mod \ell^kM_n.
 \end{equation}
Now all we have to do is to find 
some $x\in M$ such that for all $k\geq 1$, 
$x \equiv y^{(k)} \mod \ell^kM $.\par 
First, we define $x = (x_n)$ by 
\[ x_n = (y_n^{(n)})_{n\geq 1} .\]
Thus we obtain: 
\[d_n(x_{n+1}) = d_n(y^{(n+1)}_{n+1}) = y^{(n+1)}_n \equiv y^{(n)}_n \mod \ell^nM_n,\]
the last congruence being a consequence of \eqref{congruence}.
But since $M_n$ is a $\Z{\ell^n}$-module, $l^nM_n=\{0\}$, thus 
$d_n(x_{n+1})=  y^{(n)}_n = x_n$. Hence $(x_n) \in M$.\par
It is now sufficient to show that $x \equiv y^{(k)} \mod \ell^kM $ for all $k\geq 1$. 
For this, let us consider some $n\in \N^*$.
Then  
\[x_n -y^{(k)}_n = y^{(n)}_n -y^{(k)}_n. \]
If $n<k$, then according to \eqref{congruence}, 
$y^{(n)}_n -y^{(k)}_n \in \ell^n M_n$ and since 
$\ell^nM_n = \{0\}$, $y^{(n)}_n =y^{(k)}_n$, so in particular 
$y^{(n)}_n \equiv y^{(k)}_n \mod \ell^k M_k$. 
If $n\geq k$, still according to 
\eqref{congruence}, 
$y^{(n)}_n - y^{(k)}_n \in \ell^kM_k$. So in any case 
$(x - y^{(k)})_n \in \ell^kM_n$. But since the groups 
$M_n$ are all finite, according to the Mittag Leffler condition, 
\[\ell^k \varprojlim_{n\geq 1}  M_n \simeq  \varprojlim_{n\geq 1} \ell^kM_n.\]
Hence $x - y^{(k)} \in \ell^kM$ which concludes the proof. 
\end{proof}

\begin{lemme}
\label{propf}
Let $M$ be a complete and separated 
$\mathbb{Z}_\ell$-module. Then $M$ is finitely generated if and only if 
$M/\ell M$ is finite.
\end{lemme}
\begin{proof}
First, if $M$ is a finitely generated $\Zl$-module,  $M/\ell M$ is a finitely 
generated $\Z{\ell}$-module, hence is finite. \par
Conversely, if $M/\ell M$ is generated by some elements $m_1 , \ldots , m_N$ from $M$, 
we show by induction on $n$ that for each $n\geq 0$, 
$\ell^nM / (\ell^{n+1}M)$ is generated by 
$\ell^n(m_1,\ldots m_N)$. 
Indeed, this is true by hypothesis for $n=0$. 
Now, if $n>0$, and $x \in \ell^nM$, say 
$x=\sum\limits_{i=1}^N \ell^n x_i$, then $x=\ell \sum\limits_{i=1}^N \ell^{n-1}x_i$ and 
by induction hypothesis, there exists $y \in \ell^{n-1}(m_1 \ldots m_N)$ such that 
$ \sum\limits_{i=1}^N \ell^{n-1} x_i \equiv y  \mod \ell^nM$. 
Hence $\ell y \in \ell^n(m_1\ldots m_N)$ and 
$\ell y \equiv \sum\limits_{i=1}^N \ell^nx_i \mod \ell^{n+1}M$.\par
Hence if $x \in M$, one can inductively define a sequence 
$(x_n)_{n\geq 0}$ such that 
$x_n \in (m_1 \ldots m_N)$, $x_n=x$ mod $\ell^n M$ and 
$x_{n+1}-x_n \in \ell^n (m_1 \ldots m_n)$.
Hence in $\mathbb{Z}_\ell (m_1 \ldots m_N)$, $(x_n)$ has a limit which is $x$.

\end{proof}

\begin{prop}
\label{prop_finitu}
The groups 
$H^q_c(S , \mathbb{Z}_\ell ) $ are finitely generated $\mathbb{Z}_\ell$-modules. 
Hence,  $H^q_c(S , \mathbb{Q}_\ell )$ is a finitely generated vector space for all $q$, and 
 $H^q_c(S , \mathbb{Q}_\ell )=\{0\}$ for $q>2d$, where $d$ is the dimension of $\mathcal{X}^{an}$.
\end{prop}
\begin{proof}
According to proposition \ref{finisemialg} and \ref{propJ}, 
$H^q_c(S , \mathbb{Z}_\ell ) $ is a 
complete $\mathbb{Z}_\ell$-module. 
So according to lemma \ref{propf}, it only remains to prove that 
$H^q_c(S , \mathbb{Z}_\ell ) / \ell H^q_c(S , \mathbb{Z}_\ell )$ is finite. Let us prove this.\par 

For each $n\geq 0$ we have the exact sequence of groups 
\[0 \to \Z{\ell^n} \xrightarrow[]{\mu_n} \Z{\ell^{n+1} } \xrightarrow[]{\pi} \Z{\ell} \to 0  \]
where 
\[\begin{array}{rccc}
\mu_n : & \Z{\ell^n} & \to & \Z{\ell^{n+1} } \\
      & x \ \text{mod} \ \ell^n & \mapsto & \ell x \ \text{mod} \ \ell^{n+1}
\end{array}\]
and $\pi$ is the reduction map. 
If we take the long exact sequence in cohomology associated to this, we get 
the long exact sequence of projective systems:
\[\xymatrix{ 
   &  &  &  & \\
 \ar@{.>}[r]  &  \Coh \ar[r]^{\mu_n}  \ar@{.>}[u]        &  \Coho{\ell^{n+1}} \ar[r]    \ar@{.>}[u]        & \Coho{\ell}  \ar[r]  \ar@{.>}[u]   & H^{i+1}_c(S, \mathbb{Z} / \ell^n \mathbb{Z} ) \ar@{.>}[u]  \ar@{.>}[r] &          \\
 \ar@{.>}[r] &  \Coho{\ell^{n+1}} \ar[u]^{\alpha} \ar[r]^{\mu_{n+1}} & \Coho{\ell^{n+2}} \ar[u]^{\beta} \ar[r]  & \Coho{\ell} \ar[r] \ar@{=}[u] & 
H^{i+1}_c(S, \mathbb{Z} / \ell^{n+1} \mathbb{Z} ) \ar@{.>}[u] \ar@{.>}[r] &     \\
     &   \ar@{.>}[u]    &  \ar@{.>}[u]    & \ar@{.>}[u] & \ar@{.>}[u]   &  
}\]
where the first two vertical arrows $\alpha$ and $\beta$ correspond to the natural projections. 
In addition, one checks that the composite 
$\mu_n \circ \alpha$ is just multiplication by $\ell$. 
By the previous section, the groups are 
all finite, so the functor 
$\displaystyle \varprojlim_{n\geq 0}$ is exact (this is a particular case of the Mittag-Leffler condition). 
So, once we apply the functor $\displaystyle \varprojlim_{ n\geq 0}$,  
we obtain the long exact sequence:
\[ \xymatrix{
\ar@{.>}[r] & \Co  \ar[r]^{\times \ell } & \Co \ar[r]  & \Coho{\ell} \ar[r] &
H^{i+1}_c(S,\mathbb{Z}_\ell ) \ar@{.>}[r] & 
}\]
In particular , this implies that we have an injection:
\[0 \rightarrow H^q_c(S , \mathbb{Z}_\ell ) /(\ell.H^q_c(S , \mathbb{Z}_\ell )) \rightarrow 
H^q_c(S , \Z{\ell})\]
and since $H^q_c(S , \Z{\ell})$ is finite, we conclude that $H^q_c(S , \mathbb{Z}_\ell ) /(\ell.H^q_c(S , \mathbb{Z}_\ell ) )$ is finite. 
\end{proof}
Using again the exactness of $\varprojlim$ when all the groups are finite, 
and the long exact sequence \eqref{sel}, when $S$ is a locally closed semi-algebraic subset, 
$V\subseteq S$ is a semi-algebraic subset which is open in $S$ and $F = S \setminus V$, then we get 
a long exact sequence 
\[\cdots \to H^i_c(V , \mathbb{Q}_\ell ) \to H^i_c(S , \mathbb{Q}_\ell ) \to H^i_c(F , \mathbb{Q}_\ell ) \to  \cdots \]

\subsection{K\"unneth Formula}

\begin{defi}
Let $\Lambda$ be a ring.
A complex $M^{\bullet}$ of 
$\Lambda$-modules is called \emph{strictly perfect} if it is bounded, and for all $n$, $M^n$ is a finitely generated projective $\Lambda$-module.
\end{defi}

\begin{prop}
\label{kun}
 Let 
\[\xymatrix{
  & (X\times_S Y,R\times_S T)= (X,R)\times_S(Y,T)  \ar[dd]^{h} \ar[dl]^{g'} \ar[dr]^{f'} & \\
(X,R)  \ar[dr]^f & &(Y,T)\ar[dl]^g \\
  & S &
}\]
be a cartesian square of $k$-germs, where $R$ (resp. $T$) is locally closed in $X$ (resp. in $Y$), $X,Y$ and $S$ being some Hausdorff 
$k$-analytic spaces.
Let $\mathcal{F} \in \mathcal{D}^-( X , \Zln )$ and $\mathcal{G} \in \mathcal{D}^- (Y , \Zln ) $.
Then there is a canonical isomorphism: 
\[\mathbf{R}f_! \mathcal{F} \CB \mathbf{R}g_! \mathcal{G} \simeq 
\mathbf{R}h_! \left( (g'^*\mathcal{F}) \CB (f'^* ( \mathcal{G} ) \right) \]
\end{prop}

\begin{proof}
Since $R$ is locally closed, $R = U \cap F$ where $U$ is an open subset of $X$, and $F$ a closed subset, so 
$R$ is closed in $U$, and since the inclusion 
$(U,R) \to (X,R)$ is a quasi-immersion, replacing $X$ by $U$, 
we can assume that 
$R$ is closed in $X$. 
Let us set 
$U:= X \setminus R$ the complementary open subset. \par
In a first step, let us assume that 
$T=Y$, that is to say that 
$(Y,T) =(Y,Y) \simeq Y$. 
Remark that $(X,U)$ and $U$ are isomorphic as $k$-$\mathscr{G}$erms. 
We then consider the following three cartesian diagrams: 
$$\xymatrix{
  & (X,R)\times_S Y  \ar[dd]^{h} \ar[dl]_{g'} \ar[dr]^{f'} & \\
(X,R)  \ar[dr]_f & &Y\ar[dl]^g \\
  & S &
}$$

$$\xymatrix{
  & U\times_S Y  \ar[dd]^{h_U} \ar[dl]_{g_U'} \ar[dr]^{f_U'} & \\
 U  \ar[dr]_{f_U} & &Y\ar[dl]^{g_U} \\
  & S &
}$$

$$\xymatrix{
  & X\times_S Y  \ar[dd]^{h_X} \ar[dl]_{g_X'} \ar[dr]^{f_X'} & \\
X  \ar[dr]_{f_X} & &Y\ar[dl]^{g_X} \\
  & S &
}$$

We then obtain a commutative diagram of distinguished triangles: 
\[\xymatrix{
\left( \mathbf{R}{f_U}_! ( \mathcal{F}_{|U} ) \right) \CB ( \mathbf{R}g_!  \mathcal{G} ) \ar[r]^1 \ar[d] &
\mathbf{R}{h_U}_! \left( g_U'^*( \mathcal{F}_{|U} )\CB (f_U'^* (\mathcal{G}) ) \right) \ar[d] \\
( \mathbf{R}{f_X}_! (\mathcal{F}) ) \CB ( \mathbf{R} g_! \mathcal{G}) \ar[d] \ar[r]^2 &
\mathbf{R}{h_X}_! \left( g_X'^* \mathcal{F} \CB ( f_X'^*(\mathcal{G} )) \right) \ar[d] \\
\left(\mathbf{R}f_! ( \mathcal{F}_{| (X,R) } ) \right) \CB  ( \mathbf{R}g_! \mathcal{G} ) \ar[r]^3 \ar[d]^{ [-1]} & 
\mathbf{R}h_! \left( g'^*(\mathcal{F}_{|(X,R) } ) \CB (f'^* ( \mathcal{G} )) \right) \ar[d]^{[-1]} \\
 & 
}\]
According to \cite[7.7.3]{Berko93} the arrows 
$1$ and $2$ are isomorphisms. So $3$ (which is constructed in the same way as in \textit{loc.cit.}) is also an isomorphism.\par
Next, if $(Y,T)$ is a locally closed $k$-germ, as above we can 
assume that $T$ is closed in $Y$, so that if we set $V:=Y \setminus T$, 
$V$ is an open subset of $Y$ and $V \to (Y,V)$ is a quasi-immersion. 
Hence according to the first step, the proposition holds for the 
$k$-germs $(Y,V)$ and $Y$, so using again the distinguished triangle associated to 
$(Y,V)$ and $(Y,T)$, we can conclude.
\end{proof}

Exactly in the same way, 
we can generalize \cite[5.3.10]{Berko93} to $k$-germs:
 
\begin{prop}
\label{tor_dim}
 Let $\varphi :  Y \to X $ be a Hausdorff morphism of finite dimension, 
$\mathcal{G} \in D^b(Y, \Zln)$ of finite \emph{Tor}-dimension, and 
$\mathcal{F} \in D(X, \Zln)$. Let $T$ be a locally closed subspace of $Y$, and 
let us set $f = \varphi_{|(Y,T)}$. Then 
$\mathbf{R}f_! (\mathcal{G}_{|(Y,T) })$ is also of finite \emph{Tor}-dimension, and 
there is a canonical isomorphism 
$$ \mathcal{F} \CB \mathbf{R}f_! ( \mathcal{G}_{|(Y,T)} ) \simeq \mathbf{R}f_! ( f^*(\mathcal{F}) \CB \mathcal{G}_{|(Y,T)} ).$$
\end{prop}

We now apply proposition \ref{kun} to the following situation:
we assume that $S= \mathcal{M}(k)$ and we consider the 
constant sheaves 
\[ F = \underline{\Z{\ell^n}}. \]
In that case $\mathbf{R}f_! = \mathbf{R} \Gamma_c $, and 
we have the following isomorphism in 
$\mathcal{D}^-(\Zln -\text{Mod})$:
\begin{equation}
\label{kun_fini}
\mathbf{R} \Gamma_c ( (X,R) ,\Zln ) \CB \mathbf{R} \Gamma_c ((Y,T) , \Zln ) \simeq 
\mathbf{R} \Gamma_c ( (X\times Y , R\times T) , \Zln).
\end{equation}
Our goal is now to pass from \Zln \ coefficients to 
\Ql \ coefficients which is achieved in proposition \ref{prop_coef}. 
The following arguments are a rewriting of 
the exposition of the $\ell$-adic K\"unneth formula 
for \'etale cohomology of schemes made in \cite[VI 8]{Milne}. \par
Using proposition \ref{tor_dim} with 
$\mathcal{F} = \Zlnm$ and $\mathcal{G} = \Zln$ 
yields the following isomorphism in 
$D^-(\Zlnm -\text{Mod})$:
\begin{equation}
\label{changement}
\mathbf{R} \Gamma_c((X,R) , \Zln ) \CB \Zlnm \simeq 
\mathbf{R} \Gamma_c ( (X,R) , \Zlnm)
\end{equation}
In what follows, we will work with complexes 
$M^{\bullet}$ of \Zl \ (resp. \Zln ) modules. 
According to the context, 
we will either see $M^{\bullet}$ as a complex of modules, or as its image in 
the derived category 
$\mathcal{D}( \mathbb{Z}_\ell -\text{Mod})$ (resp. $\mathcal{D}( \Zln -\text{Mod})$).
For instance when we will consider projective limits 
$\displaystyle \varprojlim_n M_n^{\bullet}$, 
this will always mean that 
the $M_n^{\bullet}$'s are complexes of $\Z{\ell^n}$-modules. 
In the same way, 
if $M^{\bullet}$ and $N^{\bullet}$ are complexes of $\Zl$-modules, 
$M^{\bullet} \otimes_{\Zl} N^{\bullet}$ will denote the total tensor product of 
complexes of $\Zl$-modules, whereas 
$M^{\bullet} \overset{L}{\otimes}_{\Zl} N^{\bullet}$ will denote the total tensor 
product of $M^{\bullet}$ and $N^{\bullet}$ seen as objects of the derived category.\par
Now we need the following lemma:
\begin{lemme}
\label{limproj}
For each $n\geq 1$, let 
$A_n^{\bullet}$ and $B_n^{\bullet}$ be 
strictly perfect complexes of
$\Zln$-modules, and for each $n\geq 2$ let 
$\varphi_n : A_n^{\bullet} \to A_{n-1}^{\bullet}$ 
(resp. $\psi_n : B_n^{\bullet} \to B_{n-1}^{\bullet}$)  be a
morphism of complex of $\Zln$-modules, such that 
the canonical morphism 
$A_{n}^{\bullet} \underset{\Zln}{\otimes} \Zlnm \to A_{n-1}^{\bullet}$ 
(resp. $B_{n}^{\bullet} \underset{\Zln}{\otimes} \Zlnm \to B_{n-1}^{\bullet}$) 
is a
quasi-isomorphism. 
Then there is a canonical isomorphism in $\mathcal{D}( \mathbb{Z}_\ell -\text{Mod})$:
\[\varprojlim_{n\geq 1}  ( A_n^{\bullet}  \PT B_n^{\bullet} ) \simeq 
( \varprojlim_{n \geq 1} A_n^{\bullet} )
\underset{\mathbb{Z}_\ell}{ \overset{L}{\otimes} }   ( \varprojlim_{n \geq 1} B_n^{\bullet} )\]
\end{lemme}
 
\begin{proof}
 According to 
\cite[I 12.5]{FK}, there exists 
a strictly perfect complex $A^{\bullet}$ of 
$\mathbb{Z}_\ell$-modules and for each 
$n$ a quasi-isomorphism 
$\alpha_n : A^{\bullet} / ( \ell^n A^{\bullet} ) \to A_n^{\bullet}$ such that
for all $n$  
the following diagram commutes up to homotopy:
\begin{equation}
\label{diag1}
\begin{gathered}
\xymatrix{
A^{\bullet} / ( \ell^n A^{\bullet}) \ar[r]^{\alpha_n} \ar[d]^{red} & A_n^{\bullet} \ar[d]^{\varphi_n} \\
A^{\bullet} / ( \ell^{n-1} A^{\bullet}) \ar[r]^-{\alpha_{n-1}}  & A_{n-1}^{\bullet} 
}
\end{gathered}
\end{equation}
and likewise there exists a strictly perfect complex 
of $\Zl$-modules $B^{\bullet}$ and some quasi-isomorphisms  
$\beta_n : B^{\bullet} / ( \ell^n B^{\bullet} ) \to B_n^{\bullet}$ such that 
the following diagram commutes up to homotopy:
\begin{equation}
\label{diag2}
\begin{gathered}
\xymatrix{
B^{\bullet} / ( \ell^n B^{\bullet}) \ar[r]^{\beta_n} \ar[d]^{red} & B_n^{\bullet} \ar[d]^{\psi_n} \\
B^{\bullet} / ( \ell^{n-1} B^{\bullet}) \ar[r]^-{\beta_{n-1}}  & B_{n-1}^{\bullet} 
} 
\end{gathered}
\end{equation}
Remind that if $M$ is a \Zl -module of finite type, there is a functorial isomorphism:
\begin{equation}
\label{prop_iso}
M \Req \varprojlim_{n\geq 1} M/(\ell^nM)
\end{equation}
We then obtain the following quasi-isomorphisms: 
\begin{equation}
\label{iso1}A^{\bullet}  \Req  \varprojlim_n (A^{\bullet} / (\ell^n A^{\bullet} ))  \Req \varprojlim_n (A_n^{\bullet} ), 
\end{equation}
\noindent where the first arrow is an isomorphism of complexes according to \eqref{prop_iso} and
the second arrow is a quasi-isomorphism  
according to Mittag Leffler condition and the fact that all 
the modules involved are of finite type. 
The similar results holds for $B^{\bullet}$.\par
We then obtain the following sequence of isomorphisms 
in $\mathcal{D}( \mathbb{Z}_\ell -\text{Mod})$:
\begin{align}
\label{isoa}( \varprojlim_n A_n^{\bullet} )
\underset{\mathbb{Z}_\ell}{ \overset{L}{\otimes} }   ( \varprojlim_n B_n^{\bullet} ) & 
\Leq A^{\bullet} \underset{\mathbb{Z}_\ell}{ \overset{L}{\otimes} }  B^{\bullet}\\
\label{isoa'} & \simeq  A^{\bullet} \underset{\mathbb{Z}_\ell}{\otimes}  B^{\bullet} \\
\label{isob}& \Req \varprojlim_n \left( 
(A^{\bullet} \underset{\mathbb{Z}_\ell}{\otimes}  B^{\bullet})/ 
( \ell^n ( A^{\bullet} \underset{\mathbb{Z}_\ell}{\otimes}  B^{\bullet} ) ) 
\right) \\ 
\label{isoc} & \Req  \varprojlim_n \left( 
 ( A^{\bullet}/( \ell^n A^{\bullet}) ) \PT  ( B^{\bullet} / ( \ell^n B^{\bullet} ) ) 
 \right) \\
\label{isoc'} & \Req  \varprojlim_n \left( 
 A_n^{\bullet} \PT  ( B^{\bullet}/ ( \ell^n B^{\bullet} ) ) 
 \right)\\ 
\label{isod}  & \Leq  \varprojlim_n(A_n^{\bullet} \PT B_n^{\bullet}).
\end{align}
The isomorphism \eqref{isoa} holds thanks to \eqref{iso1}, 
\eqref{isoa'} holds because $A^{\bullet}$ and 
$B^{\bullet}$ are flat, 
\eqref{isob} is remark \eqref{prop_iso}, 
\eqref{isoc} is base change for tensor product.\par
Finally to obtain \eqref{isoc'} we take the tensor product of 
the first (resp. second)
line of diagram \eqref{diag1} with  $B^{\bullet}/ ( \ell^n B^{\bullet} )$ 
(resp. $B^{\bullet}/ ( \ell^{n-1} B^{\bullet} )$). 
The resulting diagram still commutes up to homotopy
and since $B^{\bullet}/ ( \ell^n B^{\bullet} )$ is a projective complex, 
the horizontal lines are still quasi-isomorphisms. Hence (thanks to Mittag Leffler condition), we obtain \eqref{isoc'}.\par
Similarly, for \eqref{isod}, we take the tensor product of the first (resp. second) 
line of diagram \eqref{diag2} with 
$A_n^{\bullet}$ (resp. $A_{n-1}^{\bullet}$). 
Since $A_n^{\bullet}$ is a projective complex, the 
horizontal lines remain quasi-isomorphisms and we can conclude with the same argument.   
\end{proof}

\begin{rem}
Note that we have implicitly used the following result: if 
$M_1^{\bullet}$,  $M_2^{\bullet}$ and $M_3^{\bullet}$ are bounded above
complexes of $\Lambda$-modules such that $M_3^{\bullet}$ is projective, and $f : M_1^{\bullet}  \to M_2^{\bullet}$ is a 
quasi-isomorphism, then 
$f\otimes id  : M_1^{\bullet} \otimes M_3^{\bullet}  \to M_2^{\bullet} \otimes M_3^{\bullet}$ 
is a quasi-isomorphism \cite[10.6.2]{Weib}
\end{rem}

\begin{prop}
\label{prop_coef}
 Let $(X,R)$, $(Y,T)$ be $k$-germs such that for all $n$, 
$\HXR$ and $\HYT$ have finite cohomology groups. Then the cohomology groups of $\HXY$ are also finite and for all $r\geq 0$ we have a 
canonical isomorphism:
\[ 
H^r_c \left( (X\times Y , R\times T) , \mathbb{Q}_\ell \right)
\simeq 
\bigoplus_{p+q=r} H^p_c \left((X,R) , \mathbb{Q}_\ell \right) \otimes H^q_c \left( (Y,T) , \mathbb{Q}_\ell \right) .\]
\end{prop}
\begin{proof} 
The complexes  
$\HXR$ and $\HYT$ have bounded 
cohomology groups, are of finite type by hypothesis, 
and according to proposition \ref{tor_dim} are of finite Tor-dimension, so we can 
choose some resolutions by some strictly perfect complexes of the projective systems:
\[\begin{array}{lcc}
K_n^{\bullet} & \rightarrow & \HXR \\
P_n^{\bullet} & \rightarrow & \HYT \\
Q_n^{\bullet} & \rightarrow & \HXY.
\end{array} \]
In addition, according to \eqref{kun_fini} we can find up to homotopy a quasi-isomorphism, of projective systems:
\begin{equation}
\label{iso_chan}
K_n^{\bullet} \PT P_n^{\bullet} \simeq Q_n^{\bullet}.
\end{equation}
Moreover according to \eqref{changement}, $K^{\bullet}_n$ and $P^{\bullet}_n$ 
fulfil  the hypothesis of lemma \ref{limproj}.
We then denote by 
\begin{align*}
 K^{\bullet} &= \varprojlim_n (K_n^{\bullet}) \\
 P^{\bullet} &= \varprojlim_n (P_n^{\bullet}) \\
 Q^{\bullet} &= \varprojlim_n (Q_n^{\bullet}).
\end{align*} 
Remark that (thanks to Mittag Leffler property again) 
\begin{align}
\label{Mittag1} H^p(K^{\bullet})\simeq H^p_c((X,R) , \mathbb{Z}_\ell ) \\ 
\label{Mittage2}H^p(P^{\bullet})\simeq H^p_c((Y,T) , \mathbb{Z}_\ell ) \\ 
\label{Mittage3} H^p(Q^{\bullet})\simeq H^p_c( (X\times Y , R\times T) , \mathbb{Z}_\ell ).
\end{align}
In $\mathcal{D}( \mathbb{Z}_\ell -\text{Mod})$ we consider the following 
sequence of isomorphisms:
\begin{align}
\label{isof1}
\displaystyle
K^{\bullet} \CBl P^{\bullet} & \simeq 
\varprojlim_n (K_n^{\bullet} ) \CBl  \varprojlim_n (P_n^{\bullet} ) \\ 
\label{isof2} \displaystyle & \simeq 
\varprojlim_n ( K_n^{\bullet} \PT P_n^{\bullet} ) \\
\label{isof3} & \simeq
\varprojlim_n (Q_n^{\bullet} ) \\
\label{isof4} & \simeq  Q^{\bullet}.
\end{align}
The isomorphism \eqref{isof1} holds by definition of $K^{\bullet}$ and $P^{\bullet}$, 
\eqref{isof2} holds thanks to lemma \ref{limproj},  
\eqref{isof3} is just a consequence of \eqref{iso_chan}, and 
\eqref{isof4} holds by definition of $Q^{\bullet}$.\par
We then obtain the following isomorphisms in 
$\mathcal{D}( \mathbb{Q}_\ell -\text{Mod})$:
\[ \left( K^{\bullet} \PTZlL \Ql \right) \PTQlL 
\left( P^{\bullet} \PTZlL \Ql \right) \simeq 
\left( K^{\bullet} \PTZlL P^{\bullet} \right) \PTZlL \Ql 
\simeq Q^{\bullet} \PTZlL \Ql \]
But since $\Ql$ is flat over \Zl,  
we can replace all the $\PTZlL \Ql$ by some 
$\otimes_{\Zl} \Ql$.
Finally, since $\Ql$ is a field, 
\[
H^r \Big( ( K^\bullet \otimes_{\Zl} \Ql ) \otimes_{\Ql} (P^\bullet \otimes_{\Zl} \Ql )\Big) 
\simeq \bigoplus_{p+q=r} H^p( K^\bullet \otimes_{\Zl} \Ql ) \otimes_{\Ql} 
H^q( P^\bullet \otimes_{\Zl} \Ql ).
\]
The result then follows from the isomorphisms \eqref{Mittag1}--\eqref{Mittage3}.
\end{proof}

We must mention that the proposition \ref{kun} is functorial in $S$. 
So let  
$(\X,R)$, $(\Y,T)$ be $k$-germs, with 
$\X$ (resp. $\Y$) a separated $\A$-scheme (resp. $\mathcal{B}$-scheme) of finite type, 
$\A$ (resp. $\mathcal{B}$) being some $k$-affinoid algebra and 
$R$ (resp. $T$) being some locally closed semi-algebraic subset of 
$\X^{an}$ (resp. $\Y^{an}$).  Then, there are isomorphisms of Galois modules:
\[\bigoplus_{p+q=r} H^p_c \left((\overline{\X^{an}},\overline{R}) , \mathbb{Q}_\ell \right) \otimes H^q_c \left( (\overline{\Y^{an}},\overline{T}) , \mathbb{Q}_\ell \right) 
\simeq H^r_c \left( (\overline{(\X\times \Y)^{an}} ,\overline{R\times T}),\mathbb{Q}_\ell \right).\]

\subsection{Statement of the main theorem}
We sum up all results of this section:
\begin{theo}
\label{theoprinc}
Let $k$ be a non-Archimedean complete valued field, $\A$ a $k$-affinoid algebra,  $\mathcal{X}$ a separated 
$\A$-scheme of finite type of dimension $d$, $U$ a locally closed semi-algebraic subset of 
$\mathcal{X}^{an}$, and $\ell\neq $char$(\tilde{k})$ be a prime number. 
We denote by 
$\pi: \overline{\mathcal{X}^{an}} \to \mathcal{X}^{an}$ the morphism defined in \ref{Galois} and we set $\overline{U} = \pi^{-1} (U)$.
\begin{enumerate}
 \item The groups 
$H^i_c( \overline{U} , \mathbb{Q}_\ell) $ are finite dimensional $\mathbb{Q}_\ell$-vector spaces, 
endowed with a continuous $\Gal$-action, and $H^i_c( \overline{U} , \mathbb{Q}_\ell) =0$ for 
$i>2d$.
\item  Let $V\subset U$ be a semi-algebraic subset which is open in $U$, 
and let $F=U \setminus V$. Then 
there is a $\Gal$-equivariant long exact sequence 
\[ \xymatrix{
\ar@{.>}[r] & H^i_c(\overline{V} , \mathbb{Q}_\ell ) \ar[r] &
H^i_c(\overline{U} , \mathbb{Q}_\ell ) \ar[r] & H^i_c(\overline{F} , \mathbb{Q}_\ell ) 
\ar[r] & H^{i+1}_c(\overline{V} , \mathbb{Q}_\ell ) \ar@{.>}[r] &
} \]
\item 
For all integer $n$ there are canonical $\Gal$-equivariant isomorphisms: 
\[ 
\bigoplus_{i+j=n} H^i_c \left(\overline{U} , \mathbb{Q}_\ell \right) \otimes_{\Ql} H^j_c \left( \overline{V}, \mathbb{Q}_\ell \right) 
\simeq H^n_c \left( \overline{U\times V} , \mathbb{Q}_\ell \right).
\]
\end{enumerate}
\end{theo}

\section{Analogous statements for adic spaces}
\label{huber}
In this section, $k$ will be a non-Archimedean algebraically closed non-trivially valued field, 
$\Lambda$ a finite group prime to the characteristic of 
$\tilde{k}$ and $\A$ a strictly $k$-affinoid algebra. \par
We want to stress out that in the previous sections, 
instead of working with the \'etale cohomology developed by Berkovich in 
\cite{Berko93}, we could also have used the theory of adic spaces and its \'etale 
cohomology theory, 
developed by R. Huber in
\cite{Hu96}. This might be interesting because this will define different groups (cf. remark \ref{diff}) and will apply for 
different semianalytic (resp. semi-algebraic) subsets (cf. remark \ref{sub}).  
To avoid confusion, we will denote by 
$H^q_{c,ad}(X,\Lambda)$ the groups defined by the cohomology with compact support of adic spaces. \par
In this framework, the analogue of a a $k$-germ $(X,S)$ is the notion of a pseudo-adic space $(X,S)$ over Spa$(k,k^\circ)$ 
[\ibid 1.10.3]. The quasi-immersions will be replaced by locally closed embeddings [\ibid 1.10.8 (ii) ], 
and the analogue of \cite[4.3.4]{Berko93} which states that cohomology is invariant by quasi-immersion is 
\cite[2.3.8]{Hu96} which states the same thing for locally closed embeddings. In Huber's theory though, 
compactly supported cohomology isn't defined as a derived functor, 
but with some compactification, like in the \'etale cohomology of schemes. 
Nonetheless, one can 
check that if $i : (X,S) \to (Y,T)$ is a locally closed embedding with 
$i(S)=T$, then 
$H^q_{c,ad} ((X,S) , i^* (\mathcal{F} ) )  \simeq  H^q_{c,ad} ((Y,T) , \mathcal{F} )$. Indeed, in this case, 
$i_!=i_*$ is an exact functor (it induces an equivalence of categories), 
so $R^+i_! = i_! = i_*$ [ \ibid 5.4.1]. So 
$R^+i_! ( i^* \mathcal{F} ) \simeq \mathcal{F}$, from what it follows that 
$H^q_{c,ad} ((X,S) , i^* (\mathcal{F} ) )  \simeq  H^q_{c,ad} ((Y,T) , \mathcal{F} )$.

\begin{rem}
One has to keep in mind that compactly supported cohomology does not give the same groups in both theories, 
for instance if $X$ is the closed disc of radius one:
\[
\label{diff}
\begin{array}{|l|c|c|c|}
  \hline
i  & 0 & 1 & 2 \\
\hline
H^i_{c,Ber}(X, \Lambda ) & \Lambda & 0 & 0 \\
\hline
H^i_{c,ad}(X, \Lambda ) & 0 & 0 & \Lambda  \\
\hline
 \end{array}\]
\end{rem}
 
In section \ref{secfinaff}, we systematically used the long exact sequence 
\[\cdots  H^{q-1}_c(R, \Lambda )\to  H^q_c(T , \Lambda ) \to H^q_c(S, \Lambda ) \to H^q_c(R, \Lambda )\to \cdots \]
where
\begin{equation} 
\label{TT}
T= \{x \in S  \ \big|  \ |f(x)| < |g(x)| \}
\end{equation} and 
$R = S \setminus T$.
Although the closed-open long exact sequence is still valid for pseudo-adic spaces
 \cite{Hu96}[5.5.11 (iv)], $T$ as defined in \eqref{TT} 
is not an open subset of $S$ any more, so we cannot apply this long exact sequence.
In fact the typical example of an open subset of an adic space is
\begin{equation} 
\label{T} T=\{x \in S  \ \big| \ |f(x) | \leq |g(x)| \neq 0 \}.
\end{equation}
It will be then possible in that case to 
apply this long exact sequence (which includes the case $\{f\neq 0 \} = \{0 \leq |f| \neq 0 \}$).

\begin{rem}
\label{remeq}
 A subset $S$ is a finite Boolean combination of subsets of the form $\{ |f| \leq |g|\neq 0 \}$ if 
and only if it is a finite Boolean combination of subsets of the form 
$\{ |f| \leq |g| \}$.
\end{rem}
For instance, 
$\{|f|\leq |g| \} = \{ |f| \leq |g| \neq 0 \} \cup \left( \{g\neq 0 \} \cup \{f \neq 0 \} \right)^c$. \par
Let $\mathcal{A}$ be a (strictly) $k$-affinoid algebra. We will say that a subset 
$S \subset \text{Spa}(\mathcal{A}, \mathcal{A}^\circ)$ is semianalytic if 
it is a Boolean combination of subsets of the form
\[ \{x \in \text{Spa}(\mathcal{A}, \mathcal{A}^\circ) \st |f(x)| \leq |g(x)| \} \]
where $f,g \in \mathcal{A}$. This definition slightly differs from the one given for 
Berkovich spaces: here we do not allow real constants in the inequalities. 

\begin{lemme}
\label{lemmehu}
Let $X=\text{Spa}(\mathcal{A}, \mathcal{A}^\circ)$ be the affinoid 
adic space associated to $\mathcal{A}$, $S = \cap_{i=1}^n S_i$ 
where for each $i$, $S_i$ is of the form 
$S_i = \{x\in X \ \big| \ |f_i(x) | \leq |g_i(x)|  \neq 0\}$ 
or $S_i = \{x \in X \ \big| \ |f_i(x)|>|g_i(x) | \ \text{or} \ g_i(x)=0 \}$, 
with $f_i,g_i \in \mathcal{A}$. Then the groups 
$H_{c,ad}^q(S,\Lambda)$ are finite.
\end{lemme}
\begin{proof}
Mimic the proof of lemma \ref{lemme1} using that 
\[ \{x\in X \ \big| \ |f_i(x)| \leq |g_i(x)| \neq 0 \}^c = \{x\in X \ \big| \ |f_i(x)| > |g_i(x)| \ \text{or} \ g_i(x)=0 \}.  \]
The key point here (that makes 
possible the base case of the induction) is that for an affinoid adic 
space $Y$, the groups 
$H^q_{c,ad}(Y,\Lambda)$ are finite \cite{HufinII}[5.1].
\end{proof}

\begin{prop}
Let $T$ be a locally closed, semianalytic 
subset of $X= \text{Spa} (\mathcal{A}, \mathcal{A}^\circ)$. 
Then the groups $H^q_{c,ad}(T , \Lambda)$ are finite.
\end{prop}
\begin{proof}
According to remark \ref{remeq}, we can assume that $T$ is a finite union 
of subsets $S$ as in lemma \ref{lemmehu}. Hence we can adapt 
the proof of lemma \ref{lemme21}.
\end{proof}
In this context, if $X$ is a quasi-separated adic space of finite type over $k$, we will say that 
$S$ is \emph{locally semianalytic} if there exists a finite affinoid covering 
$\{U_i\}$ of $X$ such that $S\cap U_i$ is semianalytic in $U_i$ for all $i$. 
Adapting the proofs of proposition \ref{propber}, we obtain:
\begin{prop}
\label{proph1}
 Let $X$ be a quasi-separated adic space of finite type over $k$, and $S$ a locally closed,  
locally semianalytic subset of $X$. Then the groups 
$H^q_{c,ad}(S, \Lambda)$ are finite.
\end{prop}

We can define similarly semi-algebraic subsets $S\subset \X^{ad}$ where $\X$ is an $\A$-scheme of finite type, like in definition \ref{defisemalg}, but without the real constants 
$\lambda$. We then obtain:
\begin{prop}
\label{proph2}
Let $\mathcal{X}$ be a  separated $\A$-scheme of finite type, $S$ a locally closed 
semi-algebraic subset of
$\mathcal{X}^{ad}$.  
Then the groups 
$H^q_{c,ad}(S, \Lambda)$ are finite.
\end{prop}

\begin{rem} 
\label{sub}
As indicated above, if $X$ is a $k$-analytic (resp. adic) affinoid space, the class of 
locally closed subspaces will be different according to the theories. To illustrate this we want to give two examples.
Let us consider $X$ the closed bidisc of radius $1$: $X = \mathcal{M}({k\{x,y\}}) $ or 
Spa($k\{x,y\}, k^\circ\{x,y\})$ according to theory we are using. Remind that a subset $U$ is locally closed if and only if $U$ is open in $\overline{U}$. \par 

\textbf{A subset which is locally closed for the topology of adic spaces but not for the Berkovich topology.} 
Let $V= \{p \in X \ \big| \ |x(p)|>|y(p)| \} \cup \{p_0\}$. 
Here $p_0$ is the rigid point corresponding 
to the origin.
Then $V$ is closed in the adic topology. Indeed its complement is 
\[ V^c = \{ p\in X\setminus \{p_0\} \st |x(p)| \leq |y(p)| \}
= \{ p\in X \st |x(p)| \leq |y(p)| \neq 0\} \]
which is open by definition of the topology of adic spaces. 
But we claim that $V$ is not locally closed for the Berkovich topology.
To show this, for $r,s \leq 1$ let 
$\eta_{r,s} \in X$ be defined by 
\[ 
\eta_{r,s}( \sum_{i,j\in \N}  a_{i,j}x^iy^j ) = \max_{i,j\in \N}  |a_{i,j}| r^is^j .\]
Then for $r>s$, $\eta_{r,s} \in V$, and 
for $0<r\leq 1$, $\eta_{r,r} \in \overline{V} \setminus V$. Now, if $V$ was open in $\overline{V}$, since 
it contains $p_0$, it should contain  
$\eta_{r,r}$ for $r$ small enough which is a contradiction.\par 

\textbf{A subset which is locally closed for the Berkovich topology but not for the topology of adic spaces.}
Let us consider the set 
$U=\{ p\in X \ \big| \ |x(p)| \leq |y(p)| \}$. Then 
$U$ is closed for the Berkovich topology but not locally-closed for the topology of adic spaces.
Indeed, if $p_0$ is the rigid point corresponding to the origin $(0,0)$, $p_0 \in U$, but $U$ is not a neighbourhood of 
$p_0$ in $\overline{U}$ with respect to  the topology of adic spaces. 
Otherwise for some $\varepsilon >0$, $U$ would contain a subset 
$B =\{ p\in \overline{U} \ \big| \ |x(p)| \leq \varepsilon \ \text{and} \ |y(p)| \leq \varepsilon \}$.
But then for $0<\alpha <\varepsilon$ 
with $\alpha \in |k^{\times}|$, we can define 
$\eta_{\alpha} \in \overline{U}$ a valuation of rank $2$ such that
$\eta_{\alpha} (x) = \alpha$ and 
$\eta_{\alpha}(y) = \alpha_- $ where $\alpha_-<\alpha$ but is infinitesimally closed. 
Now, $\eta_{\alpha} \in \overline{U}$ because 
$\eta_{\alpha}\in \overline{ \{ \eta_{\alpha,\alpha} \} } $ 
(cf. definition above).
So by definition of $B$, $\eta_{\alpha} \in B$. 
So we should have $\eta_\alpha \in U$, which is false. 
So $U$ is not locally closed for the adic topology.
\end{rem}

\bibliographystyle{alpha}
\bibliography{bibli}

\end{document}